\documentclass[aop,MSNbibl,seceqn,nameyear,dvips]{arximspdf}

\doi{10.1214/12-AOP746} %kopijuoti is PTS
\volume{41}
\issue{1}
\pubyear{2013}
\firstpage{262}
\lastpage{293}

\makeatletter

\newcommand{\iint}{\int\!\!\int}

\newcommand{\cal}{\mathcal}

\newcommand{\Var}{\operatorname{Var}}

\newtheorem{theorem}{Theorem}[section]
\newtheorem{prop}{Proposition}[section]
\newtheorem{lemma}{Lemma}[section]
\newtheorem{coro}{Corollary}[section]

\newproclaim{remark}{Remark}[section]

\newcommand{\one}{{\mathbf1}}

\newcommand{\tL}{\tilde{L}}
\newcommand{\cL}{{\cal L}}

\newcommand{\hatK}{\hat{K}}

\newcommand{\D}{D}

\makeatother

\begin{document}
\begin{frontmatter}

\title{From Stein identities to moderate deviations}
\runtitle{Moderate deviations}

\begin{aug}
\author[A]{\fnms{Louis H. Y.} \snm{Chen}\thanksref{t1}\ead[label=e1]{matchyl@nus.edu.sg}},
\author[B]{\fnms{Xiao} \snm{Fang}\thanksref{t1}\ead[label=e2]{fangwd2003@gmail.com}}
\and
\author[C]{\fnms{Qi-Man} \snm{Shao}\corref{}\thanksref{t2}\ead[label=e3]{maqmshao@ust.hk}}
\runauthor{L. H. Y. Chen, X. Fang and Q.-M. Shao}
\affiliation{National University of Singapore,
National University of Singapore and
Hong Kong University of Science and Technology}
\address[A]{L. H. Y. Chen\\
Department of Mathematics\\
National University of Singapore\\
10 Lower Kent Ridge Road\\
Singapore 119076\\
Republic of Singapore\\
\printead{e1}}
\address[B]{X. Fang\\
Department of Statistics\\
\quad and Applied Probability\\
National University of Singapore\\
6 Science Drive 2\\
Singapore 117546 \\
Republic of Singapore\\
\printead{e2}}
\address[C]{Q.-M. Shao\\
Department of Mathematics\\
Hong Kong University of Science and Technology\\
Clear Water Bay, Kowloon, Hong Kong\\
China\\
\printead{e3}\\
and\\
Department of Statistics\\
Chinese University of Hong Kong\\
Shatin, N. T., Hong Kong\\
China\\
(after September 1, 2012)} %adresu isvedimo komanda gale!
\end{aug}

\thankstext{t1}{Supported in part by Grant C-389-000-010-101
from the National University of Singapore.}

\thankstext{t2}{Supported in part by Hong Kong RGC CERG---602608 and 603710.}

% HISTORY:
\received{\smonth{11} \syear{2009}}
\revised{\smonth{2} \syear{2012}}

% ABSTRACT
%
\begin{abstract}
Stein's method is applied to obtain a general Cram\'er-type moderate
deviation result for dependent random variables whose dependence is
defined in terms of a Stein identity. A corollary for zero-bias
coupling is deduced. The result is also applied to a combinatorial
central limit theorem, a general system of binary codes, the anti-voter
model on a complete graph, and the Curie--Weiss model. A general
moderate deviation result for independent random variables is also
proved.
\end{abstract}

% KEYWORDS
%

\end{frontmatter}

%s1 #&#
\section{Introduction}\label{sec1}

Moderate deviations date back to
\citet{Cr38} who obtained expansions for tail probabilities for sums
of independent
random variables about the normal distribution. For independent and
identically distributed random
variables $X_1, \ldots, X_n$ with $EX_i=0$ and $\Var(X_i)= 1$ such
that $Ee^{t_0|X_1|}\le c < \infty$ for some $t_0>0$, it follows from
Petrov [(\citeyear{Pet75}), Chapter 8, equation (2.41)] that
%
%e1.1 #&#
\begin{equation}
\label{00}
\frac{P(W_n>x)}{1-\Phi(x)}=1+O(1) \bigl(1+x^3\bigr)/\sqrt{n}
\end{equation}
for $0\le x\le a_0 n^{1/6}$, where $W_n=(X_1+\cdots+X_n)/ \sqrt{n}$
and $\Phi$ is the standard normal distribution function, $a_0>0$
depends on $c$ and $t_0$ and $O(1)$ is bounded by a constant depending
on $c$ and $t_0$. %%depends on $c$ and $t_0$.
The range $0\le x\le a_0 n^{1/6}$ and the order of the error term $O(1)
(1+x^3)/\sqrt{n}$ are optimal.

The proof of (\ref{00}) depends on the conjugate method and a
Berry--Esseen bound, while the classical proof of
Berry--Esseen bound for independent random variables uses the Fourier transform.
% of Fourier transform as in the case of the classical proof
%the Berry-Esseen theorem.
However, for dependent random \mbox{variables}, Stein's method performs much better
than the method of Fourier transform. Stein's method was introduced by
Charles Stein in 1972 and further
developed by him in 1986. Extensive applications of Stein's method to
obtain Berry--Esseen-type bounds
for dependent random variables can be found in, for example,
\citet{Dia77}, \citet{BRS89}, \citet{Bar90},
\citet{DR96}, \citet{GR97},
\citet{CS04}, \citet{Cha08} and \citet{NP09}.
Recent applications to concentration of measures and large deviations
can be found in, for example, \citet{Cha07} and \citet{ChD10}.
Expositions of Stein's method and its applications in normal and other
distributional
approximations can be found in \citet{DH04}
and \citet{CB05}.

In this paper we apply Stein's method to obtain a
Cram\'er-type moderate deviation result for dependent random variables
whose dependence is
defined in terms of an identity, called Stein identity, which plays a
central role in Stein's method. A corollary for zero-bias coupling is deduced.
The result is then applied to a combinatorial central limit theorem,
the anti-voter model, a general system of binary codes
and the Curie--Weiss model. The bounds obtained in these examples are
as in (\ref{00}) and therefore
may be optimal (see Remark~\ref{r71}). It is noted that
\citet{Raic07}
also used Stein's method
to obtain moderate deviation results for dependent random variables.
However, the dependence structure he
considered is
related to local dependence and is of a different nature from what we
assume through the Stein identity.

This paper is organized as follows. Section~\ref{sec2} is devoted to a
description of
Stein's method and to the construction of Stein identities using
zero-bias coupling
and exchangeable pairs. Section~\ref{sec3} presents a general Cram\'er-type
moderate deviation result and a corollary for zero-bias coupling.
The result is applied to the four examples mentioned above in Section
\ref{sec4}. Although the general
Cram\'er-type moderate deviation result cannot be applied directly to
unbounded independent random variables, the
proof of the general result can be adapted to prove (\ref{00}) under less
stringent
conditions, thereby extending a result of \citet{Lin61}. These are also
presented in Section~\ref{sec4}.
The rest of the paper is devoted to proofs.

%s2 #&#
\section{Stein's method and Stein's identity}\label{sec2}

Let $W$ be the random variable of interest and $Z$ be another random variable.
In approximating $\cL(W)$ by $\cL(Z)$ using Stein's method,
the difference between $Eh(W)$ and $Eh(Z)$ for a class of functions $h$
is expressed as
%
%e2.1 #&#
\begin{equation}\label{01}
Eh(W) - Eh(Z) = E\bigl\{Lf_h(W)\bigr\},
\end{equation}
where $ L$ is a linear operator and $f_h$ a bounded solution of the equation
$Lf = h - Eh(Z)$. It is known that for $N(0,1)$, $Lf(w) = f'(w) -
wf(w)$ [see \citet{Ste72}]
and for Poisson($\lambda$), $Lf(w) = \lambda f(w + 1) - w f(w)$; see
\citet{Chen75}.
However, $L$ is not unique. For example, for normal approximation $L$
can also be the generator of
the Ornstein--Uhlenbeck process, and for Poisson approximation $L$, the
generator of an immigration-death process.
The solution $f_h$ will then be expressed in terms of a Markov process.
This generator approach to Stein's method is due to
Barbour (\citeyear{Bar88}, \citeyear{Bar90}).

By (\ref{01}), bounding $Eh(W) - Eh(Z)$ is equivalent to bounding
$E\{Lf_h(W)\}$. To bound the latter one finds another operator
$\tilde{L}$ such that $E\{\tilde{L} f(W)\} = 0$,
for a class of functions $f$ including $f_h$,
and
write $\tL=L-R$ for a suitable operator $R$. The error term $E\{
Lf_h(W)\}$ is then expressed as $E \{R f_h(W)\}$.
The equation
%
%e2.2 #&#
\begin{equation}\label{02}
E\bigl\{ \tL f(W)\bigr\} = 0
\end{equation}
for a class of functions $f$ including $f_h$, is called a Stein
identity for $\cL(W)$. For normal approximation, there are four methods
for constructing a Stein identity: the direct method
[\citet{Ste72}], zero-bias coupling [\citet{GR97} and
\citet{Gol05}], exchangeable pairs [\citet{Ste86}] and Stein
coupling [\citet{CR09}]. We discuss below the construction
of Stein identities using zero-bias coupling and exchangeable pairs. As
proved in \citet{GR97}, for $W$ with $EW=0$ and $\operatorname{Var}(W)=1$,
there always exists $W^*$ such that
%
%e2.3 #&#
\begin{equation}\label{04}
E\bigl(Wf(W)\bigr) = Ef'\bigl(W^*\bigr)
\end{equation}
for all bounded absolutely continuous $f$ with bounded derivative $f'$.
The distribution of $W^*$ is called $W$-zero-biased.
If $W$ and $W^*$ are defined on the same probability space (zero-bias
coupling), we may write
$\Delta= W^* - W$. Then by (\ref{04}), we obtain the Stein identity
%
%e2.4 #&#
\begin{equation}\label{05}
E\bigl(Wf(W)\bigr) = Ef'(W + \Delta) = E \int_{-\infty}^\infty
f'(W+t) \,d \mu(t|W),
\end{equation}
where $\mu(\cdot|W)$ is the conditional distribution of $\Delta$
given $W$. Here
$\tL(w) = \int_{-\infty}^\infty f'(w+ t) \,d \mu(t |W=w) - w
f(w)$.\vspace*{1pt}

The method of exchangeable pairs [\citet{Ste86}] consists of
constructing $W'$ such that
$(W, W')$ is exchangeable. Then for any anti-symmetric function
$F(\cdot, \cdot)$, that is,
$F(w, w') = - F(w', w)$,
\[
EF\bigl(W, W'\bigr) = 0,
\]
if the expectation exists. Suppose that there exist a constant $\lambda
$ $(0 < \lambda< 1)$ and a random variable $R$
such that
%
%e2.5 #&#
\begin{equation}\label{ex-1}
E\bigl(W- W' | W\bigr) = \lambda\bigl( W - E(R|W)\bigr).
\end{equation}
Then for all $f$,
\[
E\bigl\{ \bigl(W- W'\bigr) \bigl(f(W) + f\bigl(W'
\bigr)\bigr)\bigr\} =0,
\]
provided the expectation exists.
This gives the Stein identity
%
%e2.6 #&#
\begin{eqnarray}\label{ex-2}
E\bigl(Wf(W)\bigr) & =& -\frac{ 1 }{2 \lambda} E\bigl\{ \bigl(W- W'
\bigr) \bigl(f\bigl(W'\bigr) - f(W)\bigr)\bigr\} + E\bigl(Rf(W)\bigr)
\nonumber\\[-8pt]\\[-8pt]
& =& E \int_{-\infty}^\infty f'(W+t)
\hat{K}(t) \,dt + E\bigl(R f(W)\bigr)\nonumber
\end{eqnarray}
for all absolutely continuous functions $f$ for which expectations
exist, where
$\hatK(t) = { 1 \over2 \lambda} \Delta(I(0 \leq t \leq\Delta) -
I(\Delta\leq t < 0))$ and
$\Delta= W'-W$. In this case, $\tL(w) =
\int_{-\infty}^\infty f'(w+t) E(\hat{K}(t) |W=w) \,dt + E(R|W=w) f(w)
- w f(w)$.

Both Stein identities (\ref{05}) and (\ref{ex-2}) are special cases of
%
%e2.7 #&#
\begin{equation}\label{c0}
E\bigl(Wf(W)\bigr) = E \int_{-\infty}^\infty
f'(W+t) \,d \hat{\mu} (t) + E\bigl(Rf(W)\bigr),
\end{equation}
where $\hat{\mu}$ is a random measure. We will prove a moderate
deviation result by assuming
that $W$ satisfies the Stein identity (\ref{c0}).

%s3 #&#
\section{A Cram\'er-type moderate deviation theorem}\label{sec3}

Let $W$ be a random variable of interest. Assume
that there exist a deterministic positive constant $\delta$, a random
positive measure $\hat{\mu}$ with support $[-\delta, \delta]$
and a random variable $R$ such that
%
%e3.1 #&#
\begin{equation} \label{c1}
E\bigl(Wf(W)\bigr) = E \int_{|t| \leq\delta} f'(W+t) \,d
\hat{\mu} (t) + E\bigl(Rf(W)\bigr)
\end{equation}
for all absolutely continuous function $f$ for which the expectation of
either side exists.
Let
%
%e3.2 #&#
\begin{equation}\label{1-1}
\D= \int_{|t|\leq\delta} d \hat{\mu} (t). %% \D_2 = \int_{|t|\leq
\end{equation}

%th3.1 #&#
\begin{theorem}\label{t1}
Suppose that there exist constants
$\delta_1, \delta_2$ and $\theta\geq1$ such that
%
%e3.3 #&#
%e3.4 #&#
\begin{eqnarray} \label{c2}
\bigl|E(D|W) -1\bigr| &\leq&\delta_1 \bigl(1+ |W|\bigr),
\\
\label{c3} \bigl|E(R|W)\bigr| &\leq&\delta_2 \bigl( 1+ |W|\bigr)
\quad\mbox{or}\nonumber\\[-8pt]\\[-8pt]
\bigl|E(R|W)\bigr| &\leq&
\delta_2 \bigl( 1+ W^2\bigr) \quad\mbox{and}\quad
\delta_2 |W|\le\alpha<1\nonumber
\end{eqnarray}
and
%
%e3.5 #&#
\begin{equation}\label{d2-0}
E(D |W) \leq\theta.
\end{equation}
Then
%
%e3.6 #&#
\begin{equation}\label{t1a}
\frac{ P(W > x) }{1-\Phi(x)} = 1+ O_\alpha(1) \theta^3
\bigl(1+x^3\bigr) (\delta+\delta_1+\delta_2)
\end{equation}
%
%%\red{In C-W model, it becomes $O(1) (1+x^3) \delta_2$}
for $0 \leq x \leq\theta^{-1} \min( \delta^{-1/3}, \delta_1^{-1/3},
\delta_2^{-1/3})$, where $O_\alpha(1)$ denotes a quantity whose absolute
value is bounded by a universal constant which depends on $\alpha$
only under the second alternative of (\ref{c3}).
\end{theorem}

%re3.1 #&#
\begin{remark}
Theorem~\ref{t1} is intended for bounded random variables but with
very general dependence assumptions. For this reason, the support of
the random measure $\hat{\mu}$ is assumed to be within $[-\delta,
\delta]$ where $\delta$ is typically of the order of $1/\sqrt{n}$
due to standardization. In order for the normal approximation to work,
$E(D|W)$ should be close to $1$ and $E(R|W)$ small. This is reflected
in $\delta_1$ and $\delta_2$ which are assumed to be small.
\end{remark}

For zero-bias coupling, $D=1$ and $R=0$, so conditions (\ref{c2}),
(\ref{c3}) and (\ref{d2-0}) are satisfied with
$\delta_1 = \delta_2 =0$ and $\theta=1$.
Therefore, we have:

%co3.1 #&#
\begin{coro} \label{c21}
Let $W$ and $W^*$ be defined on the same probability space satisfying
(\ref{04}). Assume that
$EW=0$, $EW^2=1$ and $|W- W^*| \leq\delta$ for some constant $\delta
$. Then
\[
\frac{ P(W \geq x) }{1-\Phi(x)} = 1+ O(1) \bigl(1+x^3\bigr) \delta
\]
for $0 \leq x \leq\delta^{-1/3}$.
\end{coro}

%re3.2 #&#
\begin{remark} \label{r2}
For an exchangeable pair $(W, W')$ satisfying (\ref{ex-1}) and
$|\Delta| \leq\delta$, (\ref{c1}) is satisfied with
$D = \Delta^2 / (2 \lambda)$.
\end{remark}

\begin{remark} \label{r5}
Although one cannot apply Theorem~\ref{t1} directly to unbounded
random variables, one can adapt the
proof of Theorem~\ref{t1} to give a proof of (\ref{00}) for independent
random variables assuming the
existence of the moment generating functions of $|X_i|^{1/2}$ thereby
extending a result of
\citet{Lin61}. This result is given in Proposition~\ref{t72}. The
proof also suggests the possibility of extending Theorem~\ref{t1} to
the case where the support of $\hat{\mu}$ may not be bounded.
%For unbounded $W$, one can apply Theorem~\ref{t1} to a truncated $W$
%and then obtain a moderate
%deviation result. However, this may not give a sharp moderate
%deviation theorem. In fact, there is no
%general optimal result even for Berry--Esseen bounds unless $W$ is of
%some specific structure. It is
%challenging to establish an optimal moderate deviation theorem for
%general $W$.
%On the other hand, although we cannot apply Theorem~\ref{t1} directly,
%we will give a new proof of \eq{00} by Stein's method, following the
%proof of Theorem~\ref{t1}.
%This also indicates the possibility of extending Theorem~\ref{t1} to
%unbounded case when $W$ has some
%special structure.
%}
\end{remark}

%s4 #&#
\section{Applications}\label{sec4}

In this section we apply Theorem~\ref{t1} to
four cases of dependent random variables, namely,
a combinatorial central limit theorem, the anti-voter model on a
complete graph, a general system of binary codes, and
the Curie--Weiss model. The proofs of
the results for the third and the fourth example will be given in the
last section. At the end of this section, we will
present a moderate deviation result for sums of independent random
variables and the proof will also be given in the
last section.

%s4.1 #&#
\subsection{Combinatorial central limit theorem}\label{sec4.1}

Let $\{a_{ij}\}_{i,j=1}^n$ be an array of real numbers satisfying
$\sum_{j=1}^n a_{ij}=0$ for all $i$
and $\sum_{i=1}^n a_{ij}=0$ for all $j$. Set $c_0=\max_{i,j} |a_{ij}|$
and $W=\sum_{i=1}^n a_{i\pi(i)}/\sigma$, where $\pi$ is a uniform
random permutation of $\{1,2,\ldots,n\}$ and $\sigma^2=E
(\sum_{i=1}^n a_{i\pi(i)})^2$. In \citet{Gol05} %\cite{Gol05},
$W$ is coupled with the zero-biased $W^*$ in such a way that
$|\Delta| = |W^*-W|\leq8c_0/\sigma$.
Therefore, by Corollary~\ref{c21} with $\delta= 8c_0/\sigma$, we have
%
%e4.1 #&#
\begin{equation}\label{010}
\frac{ P(W \geq x) }{1- \Phi(x)} = 1 + O(1) \bigl(1+x^3\bigr) c_0 /
\sigma
\end{equation}
for $ 0 \leq x \leq(\sigma/c_0)^{1/3}$.

%s4.2 #&#
\subsection{Anti-voter model on a complete graph}\label{sec4.2}

Consider the anti-voter model on a complete graph with $n$ vertices,
$1, \ldots, n$ and $(n-1)n/2$ edges. Let $X_i$ be a random variable
taking value
$1$ or $-1$ at the vertex $i$, $i =1, \ldots, n$.

Let $X=(X_1,\ldots, X_n)$, where $X_i$ takes values $1$ or $-1$.
The anti-voter model in discrete time is described as the following
Markov chain:
in each step, uniformly pick a vertex $I$ and an edge connecting
it to $J$, and then change $X_I$ to $-X_J$. Let $U=\sum_{i=1}^n X_i$
and $W=U/\sigma$, where $\sigma^2=\Var(U)$. Let $W'= (U - X_I -X_J)
/\sigma$, where $I$ is uniformly distributed on $\{1, 2,\ldots, n\}$
independent of other random variables. Consider the case where the
distribution of $X$ is the stationary distribution. Then as shown in
\citet{RR97}, $(W, W')$ is an exchangeable pair and
%
%e4.2 #&#
\begin{equation}\label{anti-1}
E\bigl(W-W'|W\bigr)=\frac{2}{n}W.
\end{equation}
According to (\ref{ex-2}), (\ref{c1}) is satisfied with $\delta= 2
/\sigma$ and $R=0$. To check conditions (\ref{c2}) and (\ref{d2-0}),
let $T$
denote the number of $1$'s among $X_1,\ldots,X_n$, $a$ be the number
of edges connecting two $1$'s, $b$ be the number of edges connecting
two $-1$'s and $c$ be the number of edges connecting $1$ and $-1$.
Since it is a complete graph, $a=\frac{T(T-1)}{2}$,
$b=\frac{(n-T)(n-T-1)}{2}$. Therefore [see, e.g., \citet{RR97}]
%
%e4.3 #&#
%e4.4 #&#
\begin{eqnarray}
\label{anti-2}
E\bigl[\bigl(W-W'\bigr)^2|X\bigr] & =&
\frac{1}{\sigma^2}E\bigl[\bigl(U'-U\bigr)^2|X\bigr]=
\frac{4}{\sigma^2}\frac{2a+2b}{n(n-1)}
\nonumber\\[-8pt]\\[-8pt]
&=&\frac{1}{\sigma^2}\frac{2U^2+2n^2-4n}{n(n-1)}=\frac{2\sigma
^2W^2+2n^2-4n}{\sigma^2n(n-1)},\nonumber
\\
\label{anti-3}
E(D|W)-1 & = & \frac{n}{4}E\bigl(\bigl(W'-W
\bigr)^2|W\bigr)-1
\nonumber\\[-8pt]\\[-8pt]
& = & \frac{W^2}{2(n-1)}-\frac{2\sigma^2(n-1)-(n^2-2n)}{2\sigma
^2(n-1)}.\nonumber
\end{eqnarray}
Noting that $E(E(D|W)-1)=0$ and $EW^2=1$, we have
$\sigma^2=\frac{n^2-2n}{2n-3}$. Hence
%
%e4.5 #&#
\begin{equation}\label{anti-4}
E(D|W)-1=\frac{W^2}{2(n-1)}-\frac{1}{2(n-1)},
\end{equation}
which means that (\ref{c2}) is satisfied with
$\delta_1=O(n^{-1/2})$. Thus, we have the following moderate deviation result.

%pr4.1 #&#
\begin{prop}\label{p22}
We have
\[
\frac{ P(W\geq x) }{1- \Phi(x)} = 1 + O(1) \bigl(1+x^3\bigr)/\sqrt{n}
\]
for $0 \leq x \leq n^{1/6}$.
\end{prop}

%s4.3 #&#
\subsection{A general system of binary codes}\label{sec4.3}

In \citet{CHSZ11}, a general system of binary codes
is defined as follows.
Suppose each nonnegative integer $x$ is coded by a binary string
consisting of $0$'s and $1$'s.
Let $\tilde{S}(x)$ denote the number of $1$'s in the resulting coding
string of $x$, and let
%
%e4.6 #&#
\begin{equation}
\tilde{\mathbf S}=\bigl(\tilde{S}(0), \tilde{S}(1), \ldots\bigr).
\end{equation}
For each nonnegative integer $n$, define $\tilde{S}_n=\tilde{S}(X)$,
where $X$ is a random integer uniformly
distributed over the set $\{0,1,\ldots,n\}$. The general system of
binary codes introduced by \citet{CHSZ11} is one in which
%
%e4.7 #&#
\begin{equation}
\label{p21cond} \tilde{S}_{2m-1}=\tilde{S}_{m-1}+\mathcal{I}
\qquad\mbox{in distribution}\qquad \mbox{for all } m\ge1,
\end{equation}
where $\mathcal{I}$ is an independent $\operatorname{Bernoulli}(1/2)$ random variable.
\citet{CHSZ11} proved the asymptotic normality of
$\tilde{S}_n$. Here, we apply Theorem~\ref{t1} to obtain the
following Cram\'er moderate deviation result. For $n\ge1$, let integer
$k$ be such that $2^{k-1}-1<n\le2^k-1$, and let $\tilde{W}_n=(\tilde
{S}_n-k/2)/\sqrt{k/4}$.
%
%pr4.2 #&#
\begin{prop} \label{p21}
Under the assumption (\ref{p21cond}), we have
%
%e4.8 #&#
\begin{equation}\label{p21a}
\frac{ P( \tilde{W}_n \geq x) }{1- \Phi(x)} = 1 + O(1) \bigl(1+x^3\bigr
) \frac{1}{\sqrt{k}}
\end{equation}
%
%{ P( W_n \le-x) \over1- \Phi(x)} = 1 + O(1) (1+x^3)
%}
for $0 \leq x \leq k^{1/6}$.
\end{prop}
As an example of this system of binary codes, we consider the binary
expansion of a random integer $X$
uniformly distributed over $\{0,1,\ldots,n\}$. For $2^{k-1}-1<n\le
2^k-1$, write $X$ as
\[
X= \sum_{i=1}^k X_i
2^{k-i},
\]
and let $S_n=X_1 + \cdots + X_k$. Set $W_n=(S_n-k/2)/\sqrt{k/4}$. It is
easy to verify that
$S_n$ satisfies (\ref{p21cond}). A Berry--Esseen bound for $W_n$ was
first obtained by
\citet{Dia77}. %~\cite{Dia77}.
Proposition~\ref{p21} provides a Cram\'er moderate deviation result
for $W_n$.
Other examples of this system of binary codes include the binary
reflected Gray code and a coding system using
translation and complementation. Detailed descriptions of these codes
are given in \citet{CHSZ11}.

%s4.4 #&#
\subsection{Curie--Weiss model}\label{sec4.4}

Consider the Curie--Weiss model for $n$ spins
$\Sigma=(\sigma_1,\sigma_2,\ldots,\sigma_n)\in\{-1,1\}^n$.
The joint distribution of $\Sigma$ is given by
\[
Z_{\beta,h}^{-1} \exp\Biggl(\frac{\beta}{n} \sum
_{1\le i<j \le n} \sigma_i \sigma_j +\beta h
\sum_{i=1}^n \sigma_i
\Biggr),
\]
where $Z_{\beta,h}$ is the normalizing
constant, and $\beta>0, h\in\mathbb{R}$ are called the inverse of
temperature and the external field,
respectively. We are interested in the total magnetization $S=\sum
_{i=1}^n \sigma_i$.
We divide the region $\beta>0, h\in\mathbb{R}$ into three parts, and
for each part, we list
the concentration property
and the limiting distribution of $S$ under proper standardization.
Consider the solution(s) to the equation
%
%e4.9 #&#
\begin{equation}
\label{CW1} m=\tanh\bigl(\beta(m+h)\bigr).
\end{equation}

\begin{longlist}[\textit{Case} 1.]
\item[\textit{Case} 1.] $0<\beta<1, h\in\mathbb{R}$ or
$\beta\ge1, h\ne0$. There is a unique solution $m_0$ to
(\ref{CW1}) such that $m_0 h \ge0 $. In this case, $S/n$ is
concentrated around $m_0$ and has a Gaussian limit under proper standardization.

\item[\textit{Case} 2.] $\beta>1, h=0$. There are two
nonzero solutions to (\ref{CW1}),
$m_1<0<m_2$, where $m_1 = - m_2$. Given condition on $S< 0$ ($S>0$,
resp.), $S/n$ is concentrated around $m_1$ ($m_2$, resp.) and has a
Gaussian limit under proper standardization.

\item[\textit{Case} 3.] $\beta=1, h=0$. $S/n$ is
concentrated around $0$, but the limit distribution is not Gaussian.
\end{longlist}

We refer to \citet{E85} for the concentration of measure results,
Ellis and Newman (\citeyear{EN78}, \citeyear{EN782}) for the results on
limiting distributions. See also \citet{ChS11} for a
Berry--Esseen-type bound when the limiting distribution is not
Gaussian. Here we focus on the Gaussian case and prove the following
two Cram\'er moderate deviation results for cases 1 and 2.

%pr4.3 #&#
\begin{prop} \label{p4}
In case 1, define
%
%e4.10 #&#
\begin{equation}
W=\frac{S-n m_0}{\sigma},
\end{equation}
where
%
%e4.11 #&#
\begin{equation}
\sigma^2=\frac{ n( 1-m_0^2)}{1-(1-m_0^2)\beta}.\vadjust{\goodbreak}
\end{equation}
Then we have
%
%e4.12 #&#
\begin{equation}\label{p4a}
\frac{P( W \geq x) }{1- \Phi(x)} = 1 + O(1) \bigl(1+x^3\bigr)/ \sqrt{n}
\end{equation}
%
%{ P( W \le-x) \over1- \Phi(x)}
%= 1 + O(1) (1+x^3)/ \sqrt{n} \label{p4b}
%}
for $ 0 \leq x \leq n^{1/6}$.
\end{prop}

%pr4.4 #&#
\begin{prop} \label{p5}
In case 2, define
%
%e4.13 #&#
\begin{equation}
W_1=\frac{S-n m_{1}}{\sigma_1},\qquad W_2=\frac{S-n m_{2}}{\sigma_2},
\end{equation}
where
%
%e4.14 #&#
\begin{equation}
\sigma_1^2=\frac{n( 1-m_1^2) }{1-(1-m_{1}^2)\beta},\qquad \sigma_2^2=
\frac{n( 1-m_2^2) }{1-(1-m_{2}^2)\beta}.
\end{equation}
Then we have
%
%e4.15 #&#
\begin{equation}\label{p5a}
\frac{P( W_1 \geq x|S<0)}{1- \Phi(x)} = 1 + O(1) \bigl(1+x^3\bigr)/
\sqrt{n}
\end{equation}
%
%{ P( W \le-x) \over1- \Phi(x)}
%= 1 + O(1) (1+x^3)/ \sqrt{n}
%for $ 0 \leq x \leq n^{1/6}$.
%}
and
%
%e4.16 #&#
\begin{equation}\label{p5b}
\frac{P( W_2 \geq x|S>0)}{1- \Phi(x)} = 1 + O(1) \bigl(1+x^3\bigr)/
\sqrt{n}
\end{equation}
for $ 0 \leq x \leq n^{1/6}$.
\end{prop}

%It is known that even for i.i.d. bounded random variables $\{X_i, 1
%the range $0\leq x \leq O(n^{1/6})$ in the Cram\'er moderate deviation
%for $W=n^{-1/2} \sum_{i=1}^n X_i$
%is largest possible. This indicates that our results in the previous
%four
%examples may be the
%best possible.
%}

%s4.5 #&#
\subsection{Independent random variables}\label{sec4.5}

Moderate deviation for independent random variables has been extensively
studied in literature [see, e.g., \citet{Pet75}, Chapter 8] based on the
conjugated method. Here, we will adapt the proof of Theorem~\ref{t1}
to prove the following moderate deviation result, which is a variant of
those in the literature [see again \citet{Pet75}, Chapter~8].

%pr4.5 #&#
\begin{prop}\label{t71}
Let $\xi_i, 1 \leq i \leq n$ be independent random variables with
$E\xi_i=0$ and $Ee^{t_n |\xi_i|} < \infty$ for some $t_n$ and for
each $1 \leq i \leq n$. Assume that
%
%e4.17 #&#
\begin{equation}\label{t71-c}
\sum_{i=1}^n E\xi_i^2
=1.
\end{equation}
Then
%
%e4.18 #&#
\begin{equation}\label{t71a}
\frac{P( W \geq x)}{1 - \Phi(x)} = 1 + O(1) \bigl(1+ x^3\bigr) \gamma
e^{4 x^3
\gamma} %%+ e^{- \gamma^{-2}}\big)
\end{equation}
for $ 0 \leq x \leq t_n$, where $\gamma= \sum_{i=1}^n E|\xi_i|^3
e^{x |\xi_i|} $.
\end{prop}

%We deduce the following proposition from Proposition~\ref{t71}.
%
%Let
%$X_i, 1 \leq i \leq n $ be a sequence of independent random
%variables with
%$EX_i=0$. Put $S_n = \sum_{i=1}^n X_i$ and $B_n^2 = \sum_{i=1}^n
%EX_i^2$. Assume that
%there exists positive constants $c_1, c_2, t_0$ such that
%B_n^2 \geq c_1^2 n, Ee^{t_0 |X_i|} \leq c_2 \mbox{for}
%1 \leq i \leq n.
%Then
%{ P( S_n / B_n \geq x) \over1- \Phi(x)}=
%1+O(1) (1+x^3)/\sqrt{n}
%for $0\le x \le n^{1/6}$ where $O(1)$ is an absolute constant
%depending on $c_1, c_2$ and $t_0$.
%
%{\bf Proof of Proposition~\ref{t73}}.
%Write $\xi_i=X_i/B_n$ and $t_n=n^{1/6}$. Then
%E e^{t_n |\xi_i|} =E e^{n^{1/6} |X_i|/B_n}\le E e^{|X_i|/(c_1^{1/2}
%n^{1/3})}\le c_2 <\infty
%if $n\ge1/(c_1^{3/2}t_0^3)$. Applying Proposition~\ref{t71} to $\{
%&\le& \frac{1}{c_1^{3/2} n^{3/2}} \sum_{i=1}^n (E |X_i|^6)^{1/2} (E
%e^{2|X_i|/(c_1^{1/2}n^{1/3})})^{1/2}\nonumber\\
%&\le& \frac{1}{c_1^{3/2} n^{3/2}} \sum_{i=1}^n (\frac{6!}{t_0^6} E
%e^{t_0 |X_i|})^{1/2} (E e^{t_0 |X_i|})^{1/2}\nonumber\\
%&\le& c(c_1, c_2, t_0) \frac{1}{\sqrt{n}}.
%if $n\ge8/(c_1^{3/2}t_0^3)$. Therefore, $e^{4x^3 \gamma}\le
%e^{4c(c_1,c_2,t_0)}$ for sufficiently large $n$. Thus the proposition
%is proved.
%
%As has been remarked for \eq{00} in the Introduction the range $0 \le
%x \le n^{1/6}$ and the order of the error term $O(1)(1 + x^3)/
%results in the four examples discussed above are also optimal.
%}

We deduce (\ref{00}) under less stringent conditions from Proposition
\ref{t71} and extend a result of \citet{Lin61} to independent but not
necessarily identically distributed random variables.\vadjust{\goodbreak}

%pr4.6 #&#
\begin{prop} \label{t72}
Let
$X_i, 1 \leq i \leq n $ be a sequence of independent random variables with
$EX_i=0$. Put $S_n = \sum_{i=1}^n X_i$ and $B_n^2 = \sum_{i=1}^n
EX_i^2$. Assume that
there exists positive constants $c_1, c_2$ and $t_0 $ such that
%
%e4.19 #&#
\begin{equation}\label{t72-c}
B_n^2 \geq c_1^2 n,\qquad
Ee^{t_0 \sqrt{|X_i|}} \leq c_2 \qquad\mbox{for } 1 \leq i \leq n.
\end{equation}
Then
%
%e4.20 #&#
\begin{equation}\label{t72a}
\frac{ P( S_n / B_n \geq x)}{1- \Phi(x)} = 1 + O(1) \bigl( 1+x^3\bigr)
/\sqrt{n}
\end{equation}
for $ 0 \leq x \leq(c_1 t_0^2 )^{1/3} n^{1/6}$, where $O (1)$ is
bounded by a constant depending on $c_2$ and $c_1 t_0^2 $. In
particular, we have
%
%e4.21 #&#
\begin{equation} \label{t72b}
\frac{P( S_n / B_n \geq x)}{1- \Phi(x)} \to1
\end{equation}
uniformly in $ 0 \leq x \leq o(n^{1/6})$.
\end{prop}
\begin{pf}
%{\bf Proof of Proposition~\ref{t72}}.
The main idea is first
truncating $X_i$ and then applying Proposition~\ref{t71} to the
truncated sequence. %%W.l.o.g., assume $c_1 =1$.
Let
\[
\tau_n = \bigl( c_1^2 t_0 n
\bigr)^{1/3}2^{-2/3},\qquad \bar{X}_i = X_i
\one\bigl( |X_i| \leq\tau_n^2\bigr),\qquad
\bar{S}_n = \sum_{i=1}^n
\bar{X}_i.
\]
Observe that
\begin{eqnarray*}
&&\bigl|P( S_n / B_n \geq x) - P( \bar{S}_n /
B_n \geq x)\bigr|
\\
&&\qquad\leq\sum_{i=1}^n P
\bigl(|X_i| \geq\tau_n^2 \bigr)
\\
&&\qquad\leq\sum_{i=1}^n e^{-t_0 \tau_n} E
e^{t_0 \sqrt{|X_i|}}
\leq c_2 n e^{-t_0 \tau_n}\\
&&\qquad= O(1) \bigl(1- \Phi(x)\bigr) \bigl(
1+ x^3\bigr) / \sqrt{n}
\end{eqnarray*}
for $ 0 \leq x \leq(c_1 t_0^2)^{1/3} n^{1/6} $; here we used the fact that
\[
t_0 \tau_n =\bigl(c_1
t_0^2\bigr)^{2/3} n^{1/3}
2^{-2/3}.
\]
Now let $\xi_i = ( \bar{X}_i - E\bar{X}_i) / \bar{B}_n
$,
where $\bar{B}_n^2 = \sum_{i=1}^n \operatorname{Var}(\bar{X}_i)$.
It is easy to see that
%
%e4.22 #&#
\begin{eqnarray} \label{t72-1}
\sum_{i=1}^n |E\bar{X}_i |
& \le& \sum_{i=1}^n E|X_i|
\one\bigl( |X_i| \geq\tau_n^2\bigr)
\nonumber
\\
& \leq& \sum_{i=1}^n
\sup_{s \geq\tau_n} \bigl(s^2 e^{-t_0 s}\bigr)
Ee^{t_0
\sqrt{|X_i|}}
\\
& \leq& c_2 n c_1 \bigl(c_1
t_0^2\bigr)^{-1} \sup_{s \geq t_0 \tau_n}
\bigl(s^2 e^{-s}\bigr) = c_1 o\bigl(
n^{-2}\bigr)\nonumber
\end{eqnarray}
and similarly, $\bar{B}_n = B_n( 1+ o(n^{-2}))$. Thus, for
$ 0 \leq x \leq(c_1 t_0^2)^{1/3} n^{1/6}$
\begin{eqnarray*}
x |\xi_i| &\leq&\frac{2^{1/3} x}{c_1 n^{1/2} } |X_i|\one\bigl(
|X_i|\leq\tau_n^2 \bigr) + o(1) \leq
\frac{ 2^{1/3} x \tau_n }{ c_1 n^{1/2}} \sqrt{|X_i|} + o(1) \\
&\leq&\frac
{t_0}{2^{1/3}}
\sqrt{|X_i|} + o(1)
\end{eqnarray*}
and hence $\gamma= O(n^{-1/2})$. Applying Proposition~\ref{t71}
to $\{\xi_i, 1 \leq i \leq n \}$ gives (\ref{t72a}).
\end{pf}
%re4.1 #&#
\begin{remark}\label{r71}
As stated previously for (\ref{00}) in the \hyperref[sec1]{Introduction}, the range $0
\le
x \le(c_1 t_0^2)^{1/3} n^{1/6} $ and the order of the error term
$O(1)(1 + x^3)/\sqrt{n}$ in Proposition~\ref{t72} are optimal. By
comparing with (\ref{00}) the results in the four examples discussed
above may be optimal.
\end{remark}

%s5 #&#
\section{Preliminary lemmas}\label{sec5}

To prove Theorem~\ref{t1}, we first need to develop two preliminary
lemmas. Our
first lemma gives a bound for the moment generating function of $W$.

%le5.1 #&#
\begin{lemma} \label{l21}
Let $W$ be a random variable with $E|W| \le C$. Assume that there exist
$\delta>0$, $\delta_1 \geq0, 0 \leq
\delta_2 \leq1/4$ and $\theta\geq1$
such that (\ref{c1}) and (\ref{c2})--(\ref{d2-0}) are satisfied.
Then for all $0<t \le1/(2\delta)$ satisfying
%
%e5.1 #&#
\begin{equation}
\label{l21b} t\delta_1 +C_\alpha t\theta
\delta_2 \leq1/2,
\end{equation}
where
%
%e5.2 #&#
\begin{equation}
\label{Calpha} C_\alpha= \cases{ 12, &\quad under the first alternative
of (\ref{c3}),
\vspace*{2pt}\cr
\dfrac{2(3+\alpha)}{1-\alpha}, &\quad under the second alternative of (\ref{c3}),}
\end{equation}
we have
%
%e5.3 #&#
\begin{equation}\label{l21a}
Ee^{tW} \leq\exp\bigl(t^2/2 + c_0(t)\bigr),
\end{equation}
where
%
%e5.4 #&#
\begin{equation}\label{l21c}
c_0(t) = c_1(C, C_\alpha) \theta\bigl\{
\delta_2 t + \delta_1 t^2 + (\delta+
\delta_1 + \delta_2) t^3 \bigr\},
\end{equation}
where $c_1(C, C_\alpha)$ is a constant depending only on $C$ and
$C_\alpha$.
\end{lemma}
\begin{pf}
Fix $a>0$, $t \in(0, 1/(2\delta)]$ and $s \in(0,t]$, and let
$f(w)=e^{s(w\wedge a)}$. Letting $h(s)=Ee^{s(W\wedge a)}$,
firstly we prove that $h'(s)$ can be bounded by $s h(s)$ and
$EW^2f(W)$. By (\ref{c1}),
\begin{eqnarray*}
h'(s) & = & E(W\wedge a)e^{s(W\wedge a)} \le E\bigl(Wf(W)\bigr)
\\
& = & E \int f'(W+t) \,d \hat{\mu} (t) + E\bigl(Rf(W)\bigr)
\\
& =& s E \int e^{s(W+t)} I(W+t \le a) \,d \hat{\mu} (t) + E
\bigl(e^{s(W\wedge a)} E(R|W)\bigr)
\\[-2pt]
% \label{l21-0} \\
& \le& s E \int e^{s[(W+t)\wedge a]} \,d \hat{\mu} (t) + E
\bigl(e^{s(W\wedge a)} E(R|W)\bigr)
\\[-2pt]
& \le& s E \int e^{s(W\wedge a + \delta)} \,d \hat{\mu} (t) + E \bigl
(e^{s(W\wedge a)}
E(R|W)\bigr)
\\[-2pt]
& =& s E\int e^{s(W\wedge a)} \,d \hat{\mu} (t) + s E \int e^{s(W\wedge
a)} \bigl(
e^{s \delta} - 1\bigr) \,d \hat{\mu} (t)\\[-2pt]
&&{} + E \bigl(e^{s(W\wedge a)} E(R|W)
\bigr)
\\[-2pt]
& \leq& s Ee^{s(W\wedge a)} D + s E e^{s (W\wedge a)} \bigl|e^{s \delta} -
1\bigr| D +
2 \delta_2 E \bigl(\bigl(1+W^2\bigr)e^{s(W\wedge a)}
\bigr), %\label{l21-0x}
\end{eqnarray*}
%
%%% $(W+u)\wedge a \le(W\wedge a)+\delta$
where we have applied (\ref{1-1}) and (\ref{c3}) to obtain the last
inequality. Now, applying the
simple inequality
\[
\bigl|e^x - 1\bigr| \le2 |x| \qquad\mbox{for $|x| \le1$},
\]
%
%%followed by (\ref{d2-0}),
and then (\ref{c2}), we find that
\begin{eqnarray*}
E\bigl(Wf(W)\bigr) &\le& s Ee^{s(W\wedge a)} D + s E e^{s (W\wedge a)}
2 s \delta
D+ 2\delta_2 E \bigl(\bigl(1+ W^2\bigr) e^{s(W\wedge a)}
\bigr)
\\[-2pt]
&\leq& s Ee^{s(W\wedge a)} E(D|W) + 2 s^2 \theta\delta
Ee^{s(W\wedge a)} + 2\delta_2 E \bigl(\bigl(1+W^2\bigr)
e^{s(W\wedge a)} \bigr)
\\[-2pt]
& =& s Ee^{s(W\wedge a)} + s E e^{s(W\wedge a)} \bigl[E(D|W) -1\bigr]
\\[-2pt]
&&{} + 2 s^2 \theta\delta Ee^{s(W\wedge a)} + 2\delta_2 E
\bigl(\bigl(1+W^2\bigr) e^{s(W\wedge a)} \bigr)
\\[-2pt]
& \leq& s Ee^{s(W\wedge a)} + s \delta_1 Ee^{s(W\wedge a)} \bigl(1+
|W|\bigr) + 2 s^2 \theta\delta Ee^{s(W\wedge a)}
\\[-2pt]
&&{} +2 \delta_2 E \bigl(\bigl(1+W^2\bigr)
e^{s(W\wedge a)} \bigr).
\end{eqnarray*}
Note that
%
%e5.5 #&#
\begin{eqnarray}
E|W|e^{s(W\wedge a)}&=&EW e^{s(W\wedge a)}+2 EW^- e^{s(W\wedge
a)}
\nonumber\\[-9pt]\\[-9pt]
&\le& E\bigl(Wf(W)\bigr)+ 2E|W|\le2C+E\bigl(Wf(W)\bigr).\nonumber
\end{eqnarray}
Collecting terms, we obtain
%
%e5.6 #&#
\begin{eqnarray}
\label{l21-1} &&
h'(s)\le E\bigl(Wf(W)\bigr)
\nonumber\\[-2pt]
&&\qquad\leq\bigl\{ \bigl( s (1+ \delta_1 + 2 t \theta\delta)+ 2
\delta_2\bigr) h(s) + 2 \delta_2 EW^2f(W)+2Cs
\delta_1 \bigr\}\\[-2pt]
&&\qquad\quad{} /(1-s\delta_1).
\nonumber
\end{eqnarray}

Secondly, we show that $EW^2f(W)$ can be bounded by a function of
$h(s)$ and $h'(s)$. Letting $g(w)=we^{s(w\wedge a)}$,
and then
arguing as for (\ref{l21-1}),
%
%e5.7 #&#
\begin{eqnarray}
\label{l21-02}
EW^2 f(W) &=& EW g(W)
\nonumber\\
%& = & EW^2 e^{sW} \nonumber\\
& =& E\int\bigl( e^{s[(W+t)\wedge a]} + s (W+t) e^{s[(W+t)\wedge a]} I(W+t
\le a)\bigr) \,d \hat{\mu} (t) \nonumber\\
&&{}+ E\bigl(RWf(W)\bigr)
\nonumber
\\
& \le& E\int\bigl( e^{s(W\wedge a)} e^{s\delta} + s \bigl[(W+t)\wedge a
\bigr] e^{s(W\wedge a)} e^{s\delta} \bigr) \,d \hat{\mu} (t)
\nonumber\\[-8pt]\\[-8pt]
&&{} + E\bigl(RWf(W)
\bigr)
\nonumber
\\
& =& e^{s\delta} E\bigl(f(W)+s f(W) \bigl((W\wedge a)+\delta\bigr)\bigr
) D+
E\bigl(RWf(W)\bigr)
\nonumber
\\
& \le& \theta e^{0.5} (1+0.5) Ef(W) + s\theta e^{s\delta} E(W
\wedge a)f(W) +E\bigl(RWf(W)\bigr)
\nonumber
\\
& \le& 1.5 e^ {0.5}\theta h(s)+2s\theta h'(s) +E
\bigl(RWf(W)\bigr).
\nonumber
\end{eqnarray}
Note that under the first alternative of (\ref{c3}),
%
%e5.8 #&#
\begin{equation}
\bigl|E\bigl(RWf(W)\bigr)\bigr|\le\delta_2 Ef(W) +2\delta_2
EW^2f(W),
\end{equation}
and under the second alternative of (\ref{c3}),
%
%e5.9 #&#
\begin{equation}
\bigl|E\bigl(RWf(W)\bigr)\bigr|\le\alpha E f(W) +\alpha EW^2f(W).
\end{equation}
Thus, recalling $\delta_2\le1/4$ and $\alpha<1$, we have
%
%e5.10 #&#
\begin{equation}\label{l21-3}
EW^2f(W) \leq\frac{C_\alpha}{2} \bigl(\theta h(s)+s\theta
h'(s)\bigr),
\end{equation}
where $C_\alpha$ is defined in (\ref{Calpha}).

We are now ready to prove (\ref{l21a}). Substituting (\ref{l21-3}) into
(\ref{l21-1}) yields
%
%e5.11 #&#
\begin{eqnarray} \label{l21-4}
(1-s\delta_1) h'(s) & \leq& \bigl( s (1+
\delta_1 + 2 t \theta\delta)+ 2 \delta_2\bigr) h(s)
\nonumber
\\
& &{} + \delta_2 C_\alpha\bigl(\theta h(s) +s \theta
h'(s)\bigr) +2C s\delta_1
\nonumber
\\
& = & \bigl( s (1+ \delta_1 + 2 t \theta\delta)+ 2
\delta_2(1+C_\alpha\theta) \bigr) h(s)
\nonumber\\[-8pt]\\[-8pt]
& &{} + C_\alpha s \theta\delta_2 h'(s)+2C s
\delta_1
\nonumber
\\
& \leq& \bigl( s (1+ \delta_1 + 2 t \theta\delta)+ 2
\delta_2(1+C_\alpha\theta) \bigr) h(s)
\nonumber
\\
& &{} + C_\alpha t \theta\delta_2 h'(s)+2C s
\delta_1.\nonumber
\end{eqnarray}
Solving for $h'(s)$, we obtain
%
%e5.12 #&#
\begin{equation} \label{l21-5}
h'(s)\leq\bigl( s c_1(t) + c_2(t)\bigr)
h(s) +\frac{2C s\delta_1}{1-c_3 (t)},
\end{equation}
where
%%\red{do all the following equations need labels?}
%
\begin{eqnarray*}
c_1(t) & = & \frac{ 1+ \delta_1 + 2 t \theta\delta}{1- c_3(t)}, %\label{l21-11}
\\
c_2(t) & = & \frac{ 2 \delta_2(1+C_\alpha\theta) }{1- c_3(t)}, %\label{l21-12}
\\
c_3(t) & = & t\delta_1 +C_\alpha t\theta
\delta_2. %\label{l21-15}
\end{eqnarray*}
Now taking $t$ to satisfy (\ref{l21b}) yields $c_3(t) \leq1/2$, so in
particular,
$c_i(t)$ is nonnegative for $i=1,2$ and
$1/(1-c_3(t)) \le1+ 2 c_3(t)$.

Solving (\ref{l21-5}), we have
%
%e5.13 #&#
\begin{equation}\label{l21-6}
%%\begin{equation}
h(s) \le\exp\biggl(\frac{t^2}{2}c_1(t)+t
c_2(t)+2C \delta_1 t^2\biggr).
\end{equation}
%
%%\end{equation}

Note that $ c_3(t) \leq1/2$, $\delta_2 \le1/4$ and $\theta\geq1$.
Elementary calculations now give
%%\red{this computation seems not so elementary, and the result is not
%intuitive either, so
%%we should give some detail here}
%
\begin{eqnarray*}
&&
\frac{t^2}{2}\bigl(c_1(t)-1\bigr)+tc_2(t) +2C
\delta_1 t^2
\\
&&\qquad= \frac{t^2}{2}\frac{\delta_1+2t\theta\delta
+c_3(t)}{1-c_3(t)}+\frac{2 t \delta_2(1+C_\alpha\theta
)}{1-c_3(t)}+2C
\delta_1 t^2
\\
&&\qquad\leq t^2(\delta_1+2t\theta\delta+ t
\delta_1+C_\alpha t\theta\delta_2)+4t
\delta_2(1+C_\alpha)+2C \delta_1
t^2
\\
&&\qquad\leq c_0(t),
\end{eqnarray*}
%
%%The first inequality above used the fact that $c_5(t)leq 1/2$ and the
%second inequality used the assumption that $\theta\geq1$.
and hence
\[
t^2 c_1(t) /2 + t c_2(t) + 2 C
\delta_1 t^2 \leq t^2/ 2 +
c_0(t),
\]
thus proving (\ref{l21a}) by letting $a\to\infty$.
\end{pf}

%le5.2 #&#
\begin{lemma} \label{l22}
Suppose that for some nonnegative $\delta,\delta_1$ and $\delta_2$,
satisfying
$\max(\delta, \delta_1, \delta_2) \leq1$ and $\theta\ge1$,
(\ref{l21a}) is satisfied, with $c_0(t)$ as
in (\ref{l21c}), for all
%
%e5.14 #&#
\begin{equation}
\label{MDranget} t \in\bigl[0,\theta^{-1} \min\bigl(
\delta^{-1/3}, \delta_1^{-1/3},\delta_2^{-1/3}
\bigr)\bigr].
\end{equation}
Then for integers $k \ge1$,
%
%e5.15 #&#
\begin{equation}
\label{l22a} \int_0^t u^k
e^{u^2/2} P( W \geq u) \,du \leq c_2(C, C_\alpha)
t^k,
\end{equation}
where $c_2(C, C_\alpha)$ is a constant depending only on $C$ and
$C_\alpha$ defined in Lem\-ma~\ref{l21}.
\end{lemma}
\begin{pf}
For $t$ satisfying (\ref{MDranget}), it is easy to see that $c_0(t)
\leq5c_1(C, C_\alpha)$, where $c_1(C, C_\alpha)$ is as in Lemma~\ref{l21},
%% 112 is just an upper bound, we don't need to spend time to calculate
%a sharp upper bound
%% \red{not sure where this number comes from. I am computing
%%something more in the neighborhood of 98.}
and (\ref{l21b}) is satisfied.
%% we just need to get the bound, not the constant. so 0.08 or 1 will
%not make much difference \red{With lots of room to spare? I'm getting
%% a left hand side of around.08, very much less than 1}
Write
\begin{eqnarray*}
&&\int_0^t u^k e^{u^2/2} P(
W \geq u) \,du
\\
&&\qquad= \int_0^{[t]} u^k
e^{u^2/2} P( W \geq u) \,du + \int_{[t]}^t
u^k e^{u^2/2} P( W \geq u) \,du,
\end{eqnarray*}
where $[t]$ denotes the integer part of $t$.
For the first integral, noting that $\sup_{j-1 \leq u \leq j}
e^{u^2/2- j u} = e^{(j-1)^2/2 - j(j-1)}$,
we have
%%\red{the exponent is increasing in $u$, so shouldn't one substitute
%$j$ rather
%%than $j-1$ to obtain an upper bound in the third line below? Not
%clear also about the factor of 2 that follows.}
%%%$e^{(u-j)^2/2-j^2/2}$ obtains its maximam when $u=j-1$.
%%%$\sqrt{e}<2$.
%
%e5.16 #&#
\begin{eqnarray}\label{l22-1}
&&
\int_0^{[t]} u^k e^{u^2/2}
P(W\geq u)\,du
\nonumber\\
&&\qquad\leq\sum_{j=1}^{[t]}j^k
\int_{j-1}^j e^{u^2/2-ju} e^{ju} P(W
\geq u)\,du
\nonumber
\\
&&\qquad\leq\sum_{j=1}^{[t]}j^k
e^{(j-1)^2/2-j(j-1)}\int_{j-1}^j e^{ju} P(W
\geq u)\,du
\\
&&\qquad\leq2 \sum_{j=1}^{[t]}j^k
e^{-j^2/2}\int_{-\infty}^\infty e^{ju}
P(W\geq u)\,du
\nonumber
\\
&&\qquad= 2 \sum_{j=1}^{[t]}j^k
e^{-j^2/2} (1/j) E e^{jW}
\nonumber
\\
%% && [\forall x\leq O(1)\min(\delta^{-1},\delta_1^{-1/2},
%% P(W\geq t) e^{xt} \rightarrow0 \mbox{as} t\rightarrow\infty]
&&\qquad\leq2 \sum
_{j=1}^{[t]}j^{k-1}\exp\bigl(-j^2/2
+ j^2/2 + c_0(j)\bigr)
\nonumber
\\
&&\qquad\leq2 e^{c_0(t)} \sum_{j=1}^{[t]}j^{k-1}
\nonumber
\\
&&\qquad\leq c_2(C,C_\alpha) t^k.\nonumber
\end{eqnarray}
Similarly, we have
\begin{eqnarray*}
&&
\int_{[t]}^tu^ke^{u^2/2}P(W
\geq u)\,du
\\
&&\qquad\leq t^k \int_{[t]}^t
e^{u^2/2-tu}e^{tu}P(W\geq u)\,du
\\
&&\qquad\leq t^k e^{[t]^2/2-t[t]} \int_{[t]}^t
e^{tu} P(W\geq u)\,du
\\
&&\qquad\leq2t^k e^{-t^2/2} \int_{-\infty}^\infty
e^{tu}P(W\geq u)\,du
\\
&&\qquad\leq c_2(C, C_\alpha) t^k.
\end{eqnarray*}
%
%%\[\int_{[t]}^{t} u^k e^{u^2/2} P(W\geq u)\,du \leq C (1+t^k).\]
This completes the proof.
\end{pf}

%s6 #&#
\section{Proofs of results}\label{sec6}

In this section, let $O_\alpha(1)$ denote universal constants which
depend on $\alpha$ only under the second alternative of (\ref{c3}).

%s6.1 #&#
\subsection{\texorpdfstring{Proof of Theorem \protect\ref{t1}}{Proof of Theorem 3.1}}\label{sec6.1}

%For $ 0 \leq x \leq\theta^{-1} \min( \delta^{-1/3}, \delta_1^{-1/3},
%are bounded. Moreover, $1/(1-\Phi(x)) \leq1/(1-\Phi(x_0))$ for $0
%Therefore, \eq{t1.a} is trivial if $\theta^{-1} \min( \delta^{-1/3},
%}

If $\theta^{-1} \min( \delta^{-1/3},
\delta_1^{-1/3}, \delta_2^{-1/3}) \leq O_\alpha(1)$, then $1/(1-\Phi
(x)) \leq1/(1-\Phi(O_\alpha(1)))$ for $0 \leq x \leq O_\alpha(1)$.
Moreover, $\theta^3 (\delta+\delta_1+\delta_2)\ge O_\alpha(1)$.
Therefore, (\ref{t1a}) is trivial.
Hence, we can assume
%
%e6.1 #&#
\begin{equation}\label{t1-01}
\theta^{-1} \min\bigl( \delta^{-1/3}, \delta_1^{-1/3},
\delta_2^{-1/3}\bigr) \geq O_\alpha(1)
\end{equation}
so that $\delta\le1, \delta_2\le1/4, \delta_1+2\delta_2 <1$, and
moreover, $\delta_1+\delta_2+\alpha<1$ under the second alternative
of (\ref{c3}). Our proof is based on Stein's method. Let $f=f_x$ be
the solution
to the Stein equation
%
%e6.2 #&#
\begin{equation}\label{stein2}
wf(w)- f'(w) = I( w\geq x) - \bigl(1-\Phi(x)\bigr).
\end{equation}
It is known that
%
%e6.3 #&#
\begin{eqnarray}\label{t11}
f(w) &= & \cases{ %
\sqrt{2\pi} e^{w^2/2} \bigl( 1- \Phi(w)
\bigr)\Phi(x), &\quad $w \geq x$,
\vspace*{1pt}\cr
\sqrt{2\pi} e^{w^2/2} \bigl( 1- \Phi(x)
\bigr)\Phi(w), &\quad $w < x$,}
\nonumber
\\
& \leq& \frac{ 4 }{1+w} \one(w \geq x) + 3 \bigl(1-\Phi(x)\bigr)
e^{w^2/2} \one(0 < w < x)
\\
&&{} + 4 \bigl(1-\Phi(x)\bigr) \frac{ 1 }{1+|w| } \one(w
\leq0)\nonumber
\end{eqnarray}
by using the following well-known inequality:
\[
\bigl(1- \Phi(w)\bigr) e^{w^2/2} \leq\min\biggl( \frac{ 1}{2},
\frac{ 1 }{ w
\sqrt{2 \pi}} \biggr),\qquad w >0.
\]
It is also known that $wf(w)$ is an increasing function; see
\citet{CS05}, Lemma 2.2. By (\ref{c1}) we have
%
%e6.4 #&#
\begin{equation} \label{t1-1a}
E\bigl(Wf(W)\bigr) - E\bigl(Rf(W)\bigr) = E\int f'(W+t) \,d \hat{\mu}
(t),
\end{equation}
and monotonicity of $w f(w)$ and equation (\ref{stein2}) imply that
%
%e6.5 #&#
\begin{equation} \label{t1-1b}
f'(W+t) \leq(W+\delta) f(W+\delta) + 1 - \Phi(x) - \one(W \geq x+
\delta).
\end{equation}
Recall that $\int d \hat{\mu}(t) = D$. Thus using nonnegativity of
$\hat{\mu}$ and combining
(\ref{t1-1a}) and (\ref{t1-1b}), we have
%
%e6.6 #&#
\begin{eqnarray}\label{t1-1}
&&
E \bigl(Wf(W)\bigr) - E\bigl(Rf(W)\bigr)
\nonumber
\\
&&\qquad\leq E \int\bigl((W+\delta) f(W+\delta) - W f(W)\bigr) \,d \hat
{\mu} (t) +
EWf(W) D
\\
&&\qquad\quad{} + E\int\bigl\{ 1-\Phi(x) - {\mathbf1}(W > x+ \delta) \bigr
\} \,d \hat{\mu} (t).\nonumber
\end{eqnarray}
Now, by (\ref{1-1}), the expression above can be written
%
%e6.7 #&#
\begin{eqnarray}\label{t1-2}
&&
E \bigl((W+\delta) f(W+\delta) - W f(W)\bigr)D
\nonumber
\\
& &\quad{} + EWf(W) D + E \bigl\{ 1-\Phi(x) - {\mathbf1}(W > x+ \delta)
\bigr\} D
\nonumber
\\
&&\qquad= 1- \Phi(x) - P(W > x+ \delta)
\\
&&\qquad\quad{} + E \bigl((W+\delta) f(W+\delta) - W f(W)\bigr)D +
EWf(W) D
\nonumber
\\
&&\qquad\quad{} + E \bigl\{ 1-\Phi(x) - {\mathbf1}(W > x+ \delta) \bigr\}
(D -1).\nonumber
\end{eqnarray}
Therefore, we have
%
%e6.8 #&#
\begin{eqnarray}\label{t1-2a}
&&
P(W > x+\delta) -\bigl(1- \Phi(x)\bigr)
\nonumber
\\
&&\qquad\leq E \bigl((W+\delta) f(W+\delta) - W f(W)\bigr)D + EWf(W)
(D -1)
\nonumber
\\
&&\qquad\quad{} + E \bigl\{ 1-\Phi(x) - {\mathbf1}(W > x+ \delta) \bigr\}
(D -1) + ERf(W)
\\
&&\qquad\leq\theta E \bigl((W+\delta) f(W+\delta) - W f(W)\bigr)+
\delta_1 E\bigl(|W| \bigl(1+|W|\bigr) f(W)\bigr)
\nonumber
\\
&&\qquad\quad{} + \delta_1 E \bigl| 1-\Phi(x) - {\mathbf1}(W > x+ \delta) \bigr|
\bigl(1+|W|\bigr) +
\delta_2 E\bigl(2+W^2\bigr) f(W),\nonumber
\end{eqnarray}
where we have again applied the monotonicity of $wf(w)$ as well as
(\ref{d2-0}), (\ref{c2}) and (\ref{c3}). Hence we have that
%
%e6.9 #&#
\begin{equation} \label{t1-3}
P(W > x+\delta) -\bigl(1- \Phi(x)\bigr) \le\theta I_1 +
\delta_1 I_2 + \delta_1 I_3 +
\delta_2 I_4,
\end{equation}
where
\begin{eqnarray*}
I_1 & =& E \bigl((W+\delta) f(W+\delta) - W f(W)\bigr),
\\
I_2 & =& E\bigl(|W| \bigl(1+|W|\bigr) f(W)\bigr),
\\
I_3 & =& E \bigl| 1-\Phi(x) - {\mathbf1}(W > x+ \delta) \bigr| \bigl(1+|W|\bigr)
\end{eqnarray*}
and
\[
I_4 = E\bigl(2+W^2\bigr) f(W).
\]
%
%(1- \Phi(w))e^{w^2/2} \leq\min(1/2, 1/(w \sqrt{2 \pi})) w>0
%together with the definition of $f=f_x$ in \eq{t1.1}}
By (\ref{t11}) we have
%
%e6.10 #&#
\begin{eqnarray}\label{t1-4}
Ef(W) & \leq& 4 P(W > x) + 4 \bigl(1- \Phi(x)\bigr)
\nonumber\\[-8pt]\\[-8pt]
&&{} + 3 \bigl(1- \Phi(x)\bigr) Ee^{W^2/2}{\mathbf1}(0 < W \leq x).
\nonumber
\end{eqnarray}
Note that by (\ref{c1}) with $f(w)=w$,
\begin{eqnarray*}
EW^2 & =& E\int d \hat{\mu} (t) + E(RW)
\\
&= & ED + E(RW).
\end{eqnarray*}
Therefore, under the first alternative of (\ref{c3}), $EW^2 \le
(1+2\delta_1+\delta_2) +(\delta_1+2\delta_2)EW^2$, and under the
second alternative of (\ref{c3}), $EW^2 \le(1+2\delta_1+\delta_2)
+(\delta_1+\delta_2+\alpha)EW^2$.
This shows $EW^2 \leq O_\alpha(1)$. Hence the hypotheses of Lem\-ma~\ref
{l21} is satisfied with $C=O_\alpha(1)$,
and therefore also the conclusion of Lem\-ma~\ref{l22}. In particular,
%
%e6.11 #&#
\begin{eqnarray}\label{t1-5}\quad
Ee^{W^2/2} {\mathbf1}(0 < W \leq x) %% use \leq instead of =
& \leq& P(0< W \le x)+ \int
_0^x y e^{y^2/2} P(W > y) \,dy
\nonumber\\[-8pt]\\[-8pt]
& \leq& O_\alpha(1) (1+x).\nonumber
\end{eqnarray}
Similarly, by (\ref{t11}) again,
\begin{eqnarray*}
EW^2 f(W) &\leq& 4 E|W|{\mathbf1}(W>x) + 4 \bigl(1- \Phi(x)\bigr) E|W|
\\
&&{}+ 3 \bigl(1- \Phi(x)\bigr) EW^2 e^{W^2/2}{\mathbf1}(0 < W \leq
x)
\end{eqnarray*}
and by Lemma~\ref{l22},
%
%e6.12 #&#
\begin{eqnarray} \label{t1-6}
EW^2 e^{W^2/2}{\mathbf1}(0 < W \leq x) &\le& \int
_0^x \bigl( y^3 + 2y
\bigr)e^{y^2/2} P(W > y) \,dy
\nonumber\\[-8pt]\\[-8pt]
& \leq& O_\alpha(1) \bigl( 1+ x^3\bigr).\nonumber
\end{eqnarray}
As to
\[
E|W|{\mathbf1}(W>x)\le P(W>x)+EW^2 I(W>x),
\]
it follows from Lemma~\ref{l21} that
%
%e6.13 #&#
\begin{equation}\label{t1-010}
P(W > x)\leq e^{-x^2} E e^{xW}= O_\alpha(1)
e^{-x^2/2}
\end{equation}
and
%
%e6.14 #&#
\begin{eqnarray} \label{t1-011}
\int_x^{\infty}t P(W\geq t)\,dt
&\leq& Ee^{xW} \int_x^{\infty}te^{-xt}\,dt
\nonumber\\
&=& Ee^{xW} x^{-2} \bigl(1+ x^2\bigr)
e^{-x^2} \nonumber\\[-8pt]\\[-8pt]
&\leq& O_\alpha(1) e^{-x^2/2} x^{-2}
\bigl(1+ x^2\bigr)
\nonumber\\
&\leq& O_\alpha(1) e^{-x^2/2}\nonumber
\end{eqnarray}
for $ x\geq1$. Thus
we have for $x>1$,
%
%e6.15 #&#
\begin{eqnarray}\label{t1-7}\quad
EW^2{\mathbf1}(W> x) & =& x^2 P( W > x) + \int
_x^\infty2 y P(W > y) \,dy
\nonumber\\[-8pt]\\[-8pt]
& \leq& O_\alpha(1) \bigl(1+x^2\bigr) e^{-x^2/2} \leq
O_\alpha(1) \bigl(1+ x^3\bigr) \bigl(1-\Phi(x)\bigr).
\nonumber
\end{eqnarray}
Clearly, (\ref{t1-7}) remains valid for $ 0 \leq x \leq1$ by the fact
that $EW^2 {\mathbf1}(W > x) \leq EW^2 \leq2$.
Combining (\ref{t1-5})--(\ref{t1-7}), we have
%
%e6.16 #&#
\begin{equation}\label{t1-8}
I_2 \leq O_\alpha(1) \bigl(1+x^{3}\bigr)
\bigl(1-\Phi(x)\bigr).
\end{equation}
Similarly,
%
%e6.17 #&#
\begin{equation}\label{t1-9}
I_4 \leq O_\alpha(1) \bigl(1+x^{3}\bigr)
\bigl(1-\Phi(x)\bigr)
\end{equation}
and
%
%e6.18 #&#
\begin{eqnarray}\label{t1-10}
I_3 &\leq&\bigl(1-\Phi(x)\bigr) E\bigl(2+W^2\bigr) + E
\bigl(2+W^2\bigr) {\mathbf1}(W \geq\delta+x) \nonumber\\[-8pt]\\[-8pt]
&\leq& O_\alpha(1)
\bigl(1+x^3\bigr) \bigl(1-\Phi(x)\bigr).\nonumber
\end{eqnarray}
Let $g(w)=(wf(w))'$. Then $I_1=\int_0^\delta Eg(W+t)\,dt$.
It is
easy to see that [e.g., \citet{CS01}]
%
%e6.19 #&#
\begin{equation}\label{t1-11}\quad
g(w) = \cases{ %
\bigl( \sqrt{2\pi}\bigl(1+w^2
\bigr)e^{w^2/2}\bigl(1- \Phi(w)\bigr) - w \bigr)\Phi(x), &\quad $w \geq x$,
\cr
\bigl( \sqrt{2\pi}\bigl(1+w^2\bigr)e^{w^2/2} \Phi(w) + w \bigr)
\bigl(1-\Phi(x)\bigr), &\quad $w< x$,}
\end{equation}
and
%
%e6.20 #&#
\begin{equation}\label{t1-012}
0 \leq\sqrt{2\pi}\bigl(1+w^2\bigr)e^{w^2/2}\bigl(1-\Phi(w)
\bigr) -w \leq\frac{ 2 }{
1+w^3},
\end{equation}
and we have for $0\leq t \leq\delta$,
%
%e6.21 #&#
\begin{eqnarray}
\label{t1-12}
&&
E g(W+t)
\nonumber\\
&&\qquad= Eg(W+t){\mathbf1}\{W+t\geq x\} + Eg(W+t){\mathbf1}\{W+t\leq0 \}
\nonumber
\\
&&\qquad\quad{}+ Eg(W+t){\mathbf1}\{0< W+t < x\}
\nonumber
\\
&&\qquad\leq\frac{ 2 }{1+x^3} P(W+t\geq x) + 2 \bigl(1-\Phi(x)\bigr
)P(W+t\leq0)
\\
&&\qquad\quad{} + \sqrt{2\pi} \bigl(1-\Phi(x)\bigr)\nonumber\\
&&\qquad\quad\hspace*{10.5pt}{}\times E \bigl\{ \bigl
(1+(W+t)^2+(W+t)
\bigr)e^{(W+t)^2/2} {\mathbf1}\{0< W+t < x \} \bigr\}
\nonumber
\\
&&\qquad= O_\alpha(1) \bigl(1+x^3\bigr) \bigl(1-\Phi(x)\bigr)
\nonumber
\end{eqnarray}
and hence
%
%e6.22 #&#
\begin{equation} \label{t1-12a}
I_1 = O_\alpha(1) \delta\bigl( 1+ x^3\bigr)
\bigl(1- \Phi(x)\bigr).
\end{equation}
Putting (\ref{t1-3}), (\ref{t1-8}), (\ref{t1-9}), (\ref{t1-10}) and
(\ref{t1-12a}) together gives
\[
P(W \geq x+ \delta) -\bigl(1-\Phi(x)\bigr) \leq O_\alpha(1)
\bigl(1-\Phi(x)\bigr) \theta\bigl(1+x^3\bigr) (\delta+
\delta_1+\delta_2) ,
\]
and therefore
%
%e6.23 #&#
\begin{equation}\label{t1-14}\qquad
P(W \geq x ) -\bigl(1-\Phi(x)\bigr) \leq O_\alpha(1) \bigl(1-\Phi(x)
\bigr) \theta\bigl(1+x^3\bigr) (\delta+\delta_1+
\delta_2).
\end{equation}
As to the lower bound, similarly to (\ref{t1-1b}) and (\ref{t1-2a}),
we have
\[
f'(W+t) \geq(W-\delta) f(W-\delta) + 1- \Phi(x) - \one( W \geq x -
\delta)
\]
and
\begin{eqnarray*}
&&
P(W > x -\delta) -\bigl(1- \Phi(x)\bigr)
\\
&&\qquad\geq\theta E \bigl((W-\delta) f(W-\delta) - W f(W)\bigr)-
\delta_1 E\bigl(|W| \bigl(1+|W|\bigr) f(W)\bigr)
\\
&&\qquad\quad{} - \delta_1 E \bigl| 1-\Phi(x) - {\mathbf1}(W > x- \delta) \bigr|
\bigl(1+|W|\bigr) -
\delta_2 E\bigl(2+W^2\bigr) f(W).
\end{eqnarray*}
Now follwoing the same proof of (\ref{t1-14}) leads to
% \ignore{
% $$
% P( W \geq x- \delta) - (1-\Phi(x))
% \geq- C (1-\Phi(x)) \big( (1+x^3)\delta+ (1+x^4) \delta_1 +
%(1+x^{2}) \delta_2\big)
% $$
% and
% }
%
\[
P(W \geq x ) -\bigl(1-\Phi(x)\bigr) \geq- O_\alpha(1) \bigl(1-\Phi(x)
\bigr) \theta\bigl(1+x^3\bigr) (\delta+ \delta_1 +
\delta_2).
\]
This completes the proof of Theorem~\ref{t1}.

%s6.2 #&#
\subsection{\texorpdfstring{Proof of Proposition \protect\ref{p21}}{Proof of Proposition 4.2}}\label{sec6.2}

For $n\ge2$, $X\sim U\{0,1,\ldots,n\}$, let $\tilde{S}_n=\tilde
{S}(X)$ be the number of $1$'s in the binary string of $X$ generated in
any system of binary codes satisfying (\ref{p21cond}). Without loss of
generality, assume that
%
%e6.24 #&#
\begin{equation}
\label{42-1} \tilde{S}(0)=0.
\end{equation}
Condition (\ref{p21cond}) allows $\tilde{S}(X)$ to be represented in
terms of the labels of the nodes in a binary tree described as follows.
Let $\tilde{T}$ be an infinite binary tree. For $k\ge0$, the nodes of
$\tilde{T}$ in the $k$th generation are denoted by (from left to
right) $(V_{k,0},\ldots, V_{k, 2^k-1})$. Each node is labeled by $0$
or $1$. Assume $\tilde{T}$ satisfies:
\begin{longlist}[(C3)]
\item[(C1)] the root is labeled by $0$;
\item[(C2)] the labels of two siblings are different;
\item[(C3)] infinite binary subtrees of $\tilde{T}$ with roots $\{
V_{k,0}\dvtx k\ge0\}$ are the same as $\tilde{T}$.
\end{longlist}
For $2^{k-1}-1< n\le2^{k}-1$, represent $0,\ldots,n$ by the nodes
$V_{k,0},\ldots, V_{k,n}$, respectively. Then $\tilde{S}(X)$ is the
sum of $1$'s in the shortest path from $V_{k, X}$ to the root of the
tree. Condition (C3) implies that $\tilde{S}(X)$ does not depend on $k$
so that the representation is well defined.

We consider two extreme cases. Define a binary tree $T$ by always
assigning $0$ to the left sibling and $1$ to the right sibling. Then
the number of $1$'s in the binary string of $X$ is that in the binary
expansion of $X$. Denote it by $S_n(=S(X))$. Next, define a binary tree
$\bar{T}$ by assigning $V_{k,0}=0$, $V_{k,1}=1$ for all $k$ and
assigning $1$ to the left sibling and $0$ to the right\vspace*{1pt} sibling for all
other nodes. Let the number of $1$'s in the binary string of $X$ on
$\bar{T}$ be $\bar{S}_n(=\bar{S}(X))$. Both $T$ and $\bar{T}$ are
infinity binary trees satisfying C1, C2 and C3, and both $S_n$ and
$\bar{S}_n$ satisfy (\ref{p21cond}). It is easy to see that for all
integers $n\ge0$,
%
%e6.25 #&#
\begin{equation}
S_n\le_{st} \tilde{S}_n \le_{st}
\bar{S}_n,
\end{equation}
where $\le_{st}$ denotes stochastic ordering. Therefore, it suffices
to prove Cram\'er moderate deviation results for $W_n$ and $\bar{W}_n$
where $W_n=(S_n-\frac{k}{2})/\sqrt{\frac{k}{4}}$ and $\bar
{W}_n=(\bar{S}_n-\frac{k}{2})/\sqrt{\frac{k}{4}}$. We suppress the
subscript $n$ in the following and follow \citet{Dia77} in
constructing the exchangeable pair $(W, W')$. Let $I$ be a random
variable uniformly distributed
over the set $\{1, 2, \ldots, k\}$ and independent of $X$, and let the
random variable $X'$ be defined by
\[
X' = \sum_{i=1}^k
X_i' 2^{k-i},
\]
where
%
%e6.26 #&#
\begin{equation}\label{p21-1}
X_i'= \cases{ X_i, &\quad if $i \not= I$,
\vspace*{1pt}\cr
1, &\quad if $i =I,X_I=0$ \mbox{ and } $X+2^{k-I} \leq n$,
\vspace*{1pt}\cr
0, &\quad else.}
\end{equation}
Let $S'=S-X_I+X_I'$, $W'=(S'-k/2)/\sqrt{k/4}$. As proved in
\citet{Dia77}, $(W,W')$ is an exchangeable pair and
%
%e6.27 #&#
\begin{eqnarray}
\label{p21-2}
E\bigl(W-W'|W\bigr)&=&\lambda\biggl(W-\biggl(-\frac{E(Q|W)}{\sqrt{k}}
\biggr)\biggr),
\\
%
%e6.28 #&#
\label{p21-3}
\frac{1}{2\lambda}E\bigl(\bigl(W-W'\bigr)^2|W\bigr)
-1&=&-\frac{E(Q|W)}{k},
\end{eqnarray}
where\vspace*{2pt} $\lambda=2/k$ and $Q = \sum_{i=1}^k I( X_i =0, X+ 2^{k-i} > n)$.
From Lemma~\ref{l41} and Theorem
\ref{t1} [with $\delta=O(k^{-1/2}), \delta_1 = O(k^{-1}), \delta_2
= O(k^{-1/2})$],
\[
\frac{ P( W \geq x) }{1- \Phi(x)} = 1 + O(1) \bigl(1+x^3\bigr) \frac
{1}{\sqrt{k}}
\]
for $0 \leq x \leq k^{1/6}$. Repeat the above argument for $-W$, and we have
\[
\frac{ P( W \le-x) }{1- \Phi(x)} = 1 + O(1) \bigl(1+x^3\bigr) \frac
{1}{\sqrt{k}}
\]
for $0 \leq x \leq k^{1/6}$.

Next, we notice that $S$ and $\bar{S}$ can be written as, with $X\sim
U\{0,1,\ldots, n\}$,
\[
S=I\bigl(0\le X\le2^{k-1}-1\bigr) S +I\bigl(2^{k-1}\le X \le n
\bigr) S
\]
and
\[
\bar{S}=I\bigl(0\le X\le2^{k-1}-1\bigr) \bar{S} +I\bigl(2^{k-1}
\le X \le n\bigr)\bar{S}.
\]
Therefore,
\begin{eqnarray*}
&&-W-\frac{1}{\sqrt{k/4}}
\\
&&\qquad= \biggl(-\frac{1}{2}+ I\bigl(0\le X\le2^{k-1}-1\bigr)\biggl(\frac{k-1}{2}-S\biggr)
\\
&&\qquad\quad\hspace*{42.5pt}{}+I\bigl(2^{k-1} \le X\le n\bigr) \biggl(\frac{k-1}{2}-S\biggr)
\biggr)
\Big/{\sqrt{k/4}}
\end{eqnarray*}
and
\begin{eqnarray*}
\bar{W}&=&\biggl(-\frac{1}{2}+ I\bigl(0\le
X\le2^{k-1}-1\bigr)\biggl(\bar{S}-\frac {k-1}{2}\biggr)\\
&&\hspace*{40.7pt}{} +I\bigl(2^{k-1}
\le X\le n\bigr)
\biggl(\bar{S}-\frac{k-1}{2}\biggr)\biggr)\Big/{\sqrt{k/4}}.
\end{eqnarray*}
Conditioning on $0\le X\le2^{k-1}-1$, both the distributions of $S(X)$
and $\bar{S}(X)$ are $\operatorname{Binomial}(k-1,1/2)$, which yields
\[
\mathcal{L}\biggl(\frac{k-1}{2}-S\Big| 0\le X\le2^{k-1}-1\biggr) =
\mathcal{L} \biggl(\bar{S} -\frac{k-1}{2}\Big| 0\le X\le2^{k-1}-1
\biggr).
\]
On the other hand, when $2^{k-1} \le X\le n$, $\bar{S}(X)=k-1-S(X)$.
Therefore, $\bar{W}$ has the same distribution as $-W-1/\sqrt{\frac
{k}{4}}$, which implies Cram\'er moderate deviation results also holds
for $\bar{W}$. Thus finishes the proof of Proposition~\ref{p21}.

%For each $k\ge1$, the sequence $(\tilde{S}(2^{k-1}),
%the sequence $(S(2^{k-1}), S(2^{k-1}+1), \ldots, S(2^k-1))$ in the
%following way. Step 1: divide $(S(2^{k-1}), S(2^{k-1}+1), \ldots,
%S(2^k-1))$ into two half sequences $(S(2^{k-1}), S(2^{k-1}+1), \ldots,
%S(2^{k-1}+2^{k-2}-1))$ and $(S(2^{k-1}+2^{k-2}), \ldots, S(2^k-1))$
%and choose to transpose these two subsequences or let them remain in
%the same order. Step $i, 1<i<k-1$: uniformly divide the sequence
%resulting from Step $i-1$ into $2^{i-1}$ parts, each with length
%$2^{k-i}$, then repeat Step 1 for each subsequence.
%We use the induction argument. Under the condition \eq{p2.1cond.}, $
%is transformed from the sequence $(S(2^{k-1}), S(2^{k-1}+1), \ldots,
%S(2^k-1))$ as in the lemma, we prove that $(\tilde{S}(2^{k}),
%from $(S(2^{k}), S(2^{k}+1), \ldots, S(2^{k+1}-1))$ as in the lemma.
%From condition \eq{p2.1cond.}, $\forall a_1,\ldots, a_{k-1} \in\{0,1
%&&\{\tilde{S}(2^k+a_1 2^{k-1}+\cdots+a_{k-1}2), \tilde{S}(2^k+a_1
%2^{k-1}+\cdots+a_{k-1}2+1)\}\\
%&=&\{\tilde{S}(2^{k-1}+a_1 2^{k-2}+\cdots+a_{k-1}),
%On the other hand, if $(\tilde{S}(2^{k-1}), \tilde{S}(2^{k-1}+1),
%S(2^{k-1}+1), \ldots, S(2^k-1))$ as in the lemma, and we do the same
%transformation in the first $k-1$ steps for $(S(2^{k}), S(2^{k}+1),
%of $\tilde{S}$.
%}

%le6.1 #&#
\begin{lemma}\label{l41}
We have $E(Q|S)= O(1)(1+|W|)$.
\end{lemma}
\begin{pf}
Write
\[
n = \sum_{i \ge1} 2^{k-p_i}
\]
%
%%%
%%%should be $n-1 = \sum_{i \ge1} 2^{m-p_i}$
with $1=p_1<p_2<\cdots\le p_{ k_1}$
the positions of the ones in the binary expansion of~$n$, where $k_1
\leq k$. Recall
that $X$ is uniformly distributed over $\{0, 1,\ldots,n\}$, and that
\[
X=\sum_{i=1}^k X_i
2^{k-i}
\]
with exactly $S$ of the indicator variables $X_1,\ldots,X_k$ equal to 1.

We say that $X$ falls in category $i$, $i=1,\ldots,k_1$, when
%%%
%%%the limit of $i$ may not be $m$.
%
%e6.29 #&#
\begin{equation}
\label{MDbincati} X_{p_1}=1,\qquad
X_{p_2}=1,\ldots,X_{p_{i-1}}=1
\quad\mbox{and}\quad X_{p_i}=0.
\end{equation}

We say that $X$ falls in category $k_1+1$ if $X=n$. This special
category is nonempty only when $S=k_1$, and in this case,
$Q=k-k_1$, which gives the last term in (\ref{l41-1}).

Note that if $X$ is in category $i$ for $i\le k_1$,
then, since $X$ can be no greater than~$n$, the digits of $X$ and $n$
match up to the $p_i$th,
except for the digit in place $p_i$, where $n$ has a one, and $X$ a
zero. Further, up to this digit,
$n$ has $p_i-i$ zeros, and so $X$ has $a_i=p_i-i+1$ zeros. Changing any
of these $a_i$
zeros, except the zero in position $p_i$ to ones, results in a number
$n$ or greater, while changing any other zeros,
since digit $p_i$ of $n$ is one and of $X$ zero, does not. Hence $Q$ is
at most $a_i$ when $X$
falls in category $i$.
%%\red{don't we have $Q=a_i$ in this case, why is it just a bound?}
%%% $Q$ could be $a_i-1$.
Since $X$ has $S$ ones in its expansion, $i-1$ of which are accounted
for by (\ref{MDbincati}),
the remaining $S-(i-1)$ are uniformly distributed over the\vadjust{\goodbreak}
$k-p_i=k-(i-1)-a_i$ remaining digits
$\{X_{p_i+1},\ldots,X_k\}$. Thus, we have the inequality
%
%e6.30 #&#
\begin{equation}
\label{l41-1} E(Q|S) \le\frac{1}{A}\sum_{i \ge1}
\pmatrix{k-(i-1)-a_i
\cr
S-(i-1)}a_i +
\frac{I(S=k_1)}{A}(k-k_1),
\end{equation}
where
\[
A = \sum_{i \ge1} \pmatrix{k-(i-1)-a_i
\cr
S-(i-1)}+I(S=k_1)
\]
and $1=a_1\leq a_2\leq a_3\leq\cdots\,$.

Note that if $k_1=k$, the last term of (\ref{l41-1}) equals $0$. When
$k_1< k$, we have
%
%e6.31 #&#
\begin{equation}
\frac{I(S=k_1)}{A} (k-k_1) \le\pmatrix{k-1
\cr
k_1}^{-1} (k-k_1) \le1,
\end{equation}
so we omit this term in the following argument.

We consider two cases.

\textit{Case} 1: $S\geq k/2$. As $a_i \ge1$ for all $i$, there are at
most $k+1$ nonzero terms in the sum
(\ref{l41-1}). Divide the summands into two groups, those for which
$a_i\leq2\log_2 k$ and those with $a_i> 2\log_2 k$. The first group
can sum to no more than
$2\log_2 k$
%%\red{why?}
%%
because the sum is like weighted average of $a_i$.

For the second group, note that
%
%e6.32 #&#
\begin{eqnarray} \label{l41-2}
&&\pmatrix{k-(i-1)-a_i
\cr
S-(i-1)}\Big/A
\nonumber
\\[-2pt]
&&\qquad\leq\pmatrix{k-(i-1)-a_i
\cr
S-(i-1)}\bigg/\pmatrix{k-1
\cr
S}
\nonumber\\[-9pt]\\[-9pt]
%&=& \frac{(m-S-1)\cdots(m-S-(a_i-1))}{(m-1)\cdots(m-(a_i-1))}\cdot
&&\qquad= \prod_{j=1}^{a_i-1}
\biggl( \frac{k-S-j}{k-j} \biggr) \prod_{j=0}^{i-2}
\biggl( \frac{S-j}{k-(a_i-1)-1-j} \biggr)
\nonumber
\\[-2pt]
&&\qquad\leq\frac{1}{2^{a_i-1}}\leq\frac{1}{k^2},\nonumber
\end{eqnarray}
where the second inequality follows from $S\geq k/2$,
and the last inequality from $a_i>2\log_2 k$. Therefore, the sum of
the second group of terms
is bounded by $1$.
%% \red{but to conclude that, one needs a bound on the $a_i$, the above
%%computation just bounds the coefficients that multiply $a_i$.}
%%%notice that there are at most $m$ terms in the summation and every
%$a_i$ is bounded by $m$.

\textit{Case} 2: $S< k/2$. Divide the sum on the right-hand side into two
groups according to whether $i \leq2 \log_2 k$ or $i> 2 \log_2 k$.
Clearly,
%% $\pmatrix{m-(i-1)-a_i\cr S-(i-1)}/A\leq
%%1/2^{i-1}$. \red{this inequality is not clear, please provide
%explanation. Is it
%%supposed to be true for both ranges of $i$?} Similarly to Case 1, the
%first part is bounded by $1$.
%%\red{is the first part the one mentioned second that is} $i > 2
%
\begin{eqnarray*}
&&
\pmatrix{k-(i-1)-a_i
\cr
S-(i-1)}\Big/A \\[-2pt]
&&\qquad \leq \prod
_{j=0}^{i-2} \biggl( \frac{S-j}{k-1-j} \biggr) \prod
_{j=1}^{a_i-1} \biggl( \frac{k-S-j}{k-(i-1)-j}
\biggr)
\nonumber
\\[-2pt]
&&\qquad \leq 1/2^{i-1}
\end{eqnarray*}
using the assumption $S< k/2$ and the fact that $S\geq i-1$.
The above inequality is true for all $i$, so the summation for the part
where $i> 2\log_2 k$ is bounded by $1$.

Next we consider
$i\leq2\log_2 k$. When $S\geq k ( \frac{\log a_i}{a_i-1}
) +2\log_2 k$,
we have\break
%%\red{how?}
$a_i(\frac{k-S-1}{k-(i-1)-1})^{a_i-1} \leq1$.
%%which implies that
%%%
Solving $S$ from the inequality $a_i(\frac{k-S-1}{k-(i-1)-1})^{a_i-1}
\leq1$,
we see that it is equivalent to the inequality
$S\geq(1-e^{-({\log a_i})/({a_i-1})})k-1+e^{-({\log a_i})/({a_i-1})}i$,
which is a result of the above assumption on $S$ when $i< 2\log_2 k$.
Now we have
%
%e6.33 #&#
\begin{eqnarray}\label{l41-3}
&&
a_i\pmatrix{k-(i-1)-a_i
\cr
S-(i-1)}\Big/A
\nonumber
\\
&&\qquad\leq a_i\pmatrix{k-(i-1)-a_i
\cr
S-(i-1)}\bigg/
\pmatrix{k-1
\cr
S}
\nonumber\\[-8pt]\\[-8pt]
&&\qquad= a_i \prod_{j=0}^{i-2}
\biggl( \frac{S-j}{k-1-j} \biggr) \prod_{j=1}^{a_i-1}
\biggl( \frac{k-S-j}{k-(i-1)-j} \biggr)
\nonumber
\\
%&=& \frac{S(S-1)\cdots(S-(i-1)+1)}{(k-1)\cdots(k-(i-1))}\cdot
&&\qquad\leq a_i\frac{1}{2^{i-1}}
\biggl(\frac{k-S-1}{k-(i-1)-1}
\biggr)^{a_i-1}\leq\frac{1}{2^{i-1}}\nonumber
\end{eqnarray}
%
%%\red{what happened to $a_i$ above?}
%%%
using the fact that $a_i(\frac{k-S-1}{k-(i-1)-1})^{a_i-1} \leq1$.

On the other hand, if $S< k ( \frac{\log a_i}{a_i-1} ) +
2\log_2
k$, then %%\red{why?}
%%%should be clear considering $a_i\leq m$.
$a_i S/(k-1)=O(1) \log_2 k$, which implies
%%\red{why?}
%%%
%
\begin{eqnarray*}
&&a_i \pmatrix{k-(i-1)-a_i
\cr
S-(i-1)}\Big/A \\
&&\qquad \leq
\frac{a_i S}{k-1} \prod_{j=1}^{i-2} \biggl(
\frac{S-j}{k-1-j} \biggr) \prod_{j=1}^{a_i-1}
\biggl( \frac{k-S-j}{k-(i-1)-j} \biggr)
\\
&&\qquad = O(1)\log_2 k/2^{i-2}.
\end{eqnarray*}
This proves that the right-hand side of (\ref{l41-1}) is bounded by
$O(1)\log_2
k$.

To complete the proof of the lemma, that is, to prove $E(Q|W)\leq
C(1+|W|)$, we only need to show
%%\red{why does this suffice?}
%%% Recall that we intend $
that $E(Q|S)\leq C$ for some universal constant $C$ when $|W|\leq
\log_2 k$, that is, when $k/2- \sqrt{k/4} \log_2 k \leq S\leq k/2+
\sqrt{k/4} \log_2 k $. Following the argument in case 2 above, we
only need to consider the summands where $i\leq2\log_2 k$
%%\red{why?}
because the other part where $i> 2\log_2 k$ is bounded by $1$ as
proved in case 2.

When $a_i, k$ are bigger than some universal constant, $k/2-
\sqrt{k/4} \log_2 k \geq\frac{\log a_i}{a_i-1}\times k+2\log_2 k$,
which implies $(\frac{k-S-1}{k-(i-1)-1})^{a_i-1}\times a_i\leq1$
and ${k-(i-1)-a_i\choose S-(i-1)}\times a_i/A\leq1/2^{i-1}$. Since
both parts for $i\leq2\log_2 k$ and $i> 2\log_2 k$ are bounded by
some constant, $E(Q|S)\leq C$ when $|W|\leq\log_2 k$, and hence the
lemma is proved.
\end{pf}

%s6.3 #&#
\subsection{\texorpdfstring{Proof of Propositions \protect\ref{p4} and \protect\ref{p5}}{Proof of Propositions 4.3 and 4.4}}\label{sec6.3}

Let $\tilde{W}$ have the conditional distribution of $W$ ($W_1$, $W_2$,
resp.) given $|W|\le c_1 \sqrt{n}$ ($|W_1|, |W_2| \le c_1 \sqrt{n}$,
resp.) where $c_1$ is to be determined. If we can prove that
%
%e6.34 #&#
\begin{equation}
\frac{ P(\tilde{W} \geq x) }{1-\Phi(x)} =1+ O(1) \bigl(1+x^3\bigr)
/\sqrt{n}
\end{equation}
for $0\le x \le n^{1/6}$,
then from the fact that [\citet{E85}]
%
%e6.35 #&#
\begin{equation}
\label{CW0} P\bigl(|W|>K \sqrt{n}\bigr) \le e^{-n C(K)}
\end{equation}
and
\[
P\bigl(|W_1|>K \sqrt{n}| S<0\bigr) \le e^{-n C(K)},\qquad
P\bigl(|W_2|>K \sqrt{n}| S>0\bigr) \le e^{-n C(K)}
\]
for any positive number $K$ where $C(K)$ is a positive constant
depending only on~$K$,
we have, with $\delta_2=O(1/\sqrt{n})$,
\begin{eqnarray*}
\frac{ P(W \geq x) }{1-\Phi(x)} &\le&\frac{ P(\tilde{W} \ge x)+P(\delta
_2|W|>1/2) }{1-\Phi(x)}
\\
&=& 1+ O(1) \bigl(1+x^3\bigr) /\sqrt{n}
\end{eqnarray*}
for $0\le x \le n^{1/6}$. Similarly, (\ref{p5a}) and (\ref{p5b}) are also
true. Therefore, we prove Cram\'er moderate deviation for $\tilde{W}$
(still denoted by $W$ in the following) defined below. Assume the state
space of the spins is $\Sigma=(\sigma_1, \sigma_2, \ldots, \sigma_n)\in
\{-1, 1\}^n$ such that $\sum_{i=1}^n \sigma_i/n \in[a,b]$
where $[a,b]$ is any interval within which there is only one solution
$m$ to (\ref{CW1}). Let $S=\sum_{i=1}^n \sigma_i$, $W=\frac
{S-nm}{\sigma}$ and $\sigma^2=n\frac{1-m^2}{1-(1-m^2)\beta}$. Note
that in cases 1 and 2, $1-(1-m^2)\beta>0$, thus $\sigma^2$ is
well defined. Moreover, $[a,b]$ is chosen such that $|W|\le c_1 \sqrt
{n}$. The joint distribution of the spins is
\[
Z_{\beta, h}^{-1} \exp\Biggl(\frac{\beta\sum_{1\le i<j \le n} \sigma_i
\sigma_j}{n}+\beta h\sum
_{i=1}^n \sigma_i\Biggr).
\]

Let $I$ be a random variable uniformly distributed over $\{1,
\ldots, n\}$ independent of $\{\sigma_i, 1 \leq i \leq n\}$. Let
$\sigma_i'$ be a random sample from the conditional distribution of
$\sigma_i$ given $\{\sigma_j, j \not=i, 1 \leq j \leq n\}$. Define
$W' = W- ( \sigma_I - \sigma_I')/\sigma$. Then $(W, W')$ is an
exchangeable pair.
Let
\[
A(w)=\frac{\exp(-\beta(m+h)- \beta\sigma w/n+\beta/n)}{\exp
(-\beta(m+h)- \beta\sigma w/n+\beta/n)+\exp(\beta(m+h)+ \beta
\sigma w/n-\beta/n)}
\]
and
\[
B(w)=\frac{\exp(\beta(m+h)+ \beta\sigma w/n+\beta/n)}{\exp(\beta
(m+h)+ \beta\sigma w/n+\beta/n)+\exp(-\beta(m+h)- \beta\sigma
w/n-\beta/n)}.
\]
It is easy to see that
\begin{eqnarray*}
\hspace*{-4pt}&& \frac{ e^{-\beta(m+h)- \beta\sigma w/n} }{ e^{-\beta(m+h)-\beta
\sigma w /n} + e^{\beta(m+h)+\beta\sigma w /n }}
\\
\hspace*{-4pt}&&\qquad\leq A(w) = \frac{ \exp(-\beta(m+h)-\beta\sigma w/n) }{\exp
(-\beta(m+h)-\beta\sigma w/n)+\exp(\beta(m+h)+\beta\sigma w /n - 2
\beta/n)}
\\
\hspace*{-4pt}&&\qquad\leq\frac{ e^{-\beta(m+h)- \beta\sigma w/n} }{ e^{-\beta
(m+h)-\beta\sigma w /n} + e^{\beta(m+h)+\beta\sigma w /n }}e^{2\beta
/n}
\end{eqnarray*}
and
\begin{eqnarray*}
\hspace*{-4pt}&&\frac{ e^{\beta(m+h) + \beta\sigma w/n} }{ e^{\beta(m+h)+\beta
\sigma w /n} + e^{-\beta(m+h)-\beta\sigma w /n }}
\\
\hspace*{-4pt}&&\qquad\leq B(w) = \frac{ \exp(\beta(m+h)+\beta\sigma w /n) }{\exp
(\beta(m+h)+\beta\sigma w /n)+\exp(-\beta(m+h)- \beta\sigma w /n -
2 \beta/n)}
\\
\hspace*{-4pt}&&\qquad\leq\frac{ e^{\beta(m+h) + \beta\sigma w/n} }{ e^{\beta
(m+h)+\beta\sigma w /n} + e^{-\beta(m+h)-\beta\sigma w /n
}}e^{2\beta/n}.
\end{eqnarray*}
Therefore
%%From the facts that \red{these facts also need to be supported}
%
\[
A(W)+B(W)=1+O(1)\frac{1}{n}
\]
and
\[
A(W)-B(W)=-\tanh\bigl(\beta(m+h)+\beta\sigma W/n\bigr)+O(1)
\frac{1}{n}. %\label{CW-3}
\]
Note that
\begin{eqnarray*}
&&E\bigl(W-W'|\Sigma\bigr)
\\
&&\qquad= \frac{1}{\sigma}E(\sigma_{I}-\sigma_{I}|\Sigma)
\\
&&\qquad= \frac{2}{\sigma} E\bigl(I\bigl(\sigma_I=1,
\sigma_I'=-1\bigr)-I\bigl(\sigma_I=-1,
\sigma_I'=1\bigr)|\Sigma\bigr)
\\
&&\qquad= \frac{2}{\sigma} \frac{\sigma W+nm+n}{2n} A(W) I(S-2\ge an)
\\
&&\qquad\quad{}-\frac{2}{\sigma}
\frac{n-\sigma W-nm}{2n} B(W)I(S+2\le bn)
\\
&&\qquad= \bigl(A(W)+B(W)\bigr) \biggl(\frac{W}{n}+\frac{m}{\sigma
}\biggr)+
\frac{1}{\sigma
}\bigl(A(W)-B(W)\bigr)
\\
&&\qquad\quad{}-\frac{\sigma W+nm+n}{\sigma n} A(W) I(S-2<an)\\
&&\qquad\quad{}+\frac
{n-\sigma
W-nm}{\sigma n} B(W) I(S+2>bn)
\\
&&\qquad= \biggl(\frac{W}{n}+\frac{m}{\sigma}\biggr) \biggl(1+O\biggl(
\frac{1}{n}\biggr)\biggr) -\frac
{1}{\sigma} \biggl(\tanh\biggl(
\beta(m+h)+\frac{\beta\sigma W}{n}\biggr)+O\biggl(\frac
{1}{n}\biggr)\biggr)
\\
&&\qquad\quad{}-\frac{S+n}{\sigma n}A(W)I(S-2<an)+\frac{n-S}{\sigma
n}B(W) I(S+2>bn)
\\
&&\qquad=\lambda(W-R),
\end{eqnarray*}
where
\[
\lambda=\frac{1-(1-m^2)\beta}{n}>0
\]
and
\begin{eqnarray*}
R&=&\frac{1}{\lambda}\frac{\tanh'' (\beta(m+h)+\xi) \beta^2
\sigma}{2n^2} W^2+\frac{1}{\lambda}
\frac{S+n}{\sigma n}A(W) I(S-2<an)
\\
&&{}-\frac{1}{\lambda}\frac{n-S}{\sigma n}B(W) I(S+2>bn)+O(1) \biggl(
\frac
{W}{n}+\frac{1}{\sigma}\biggr),
\end{eqnarray*}
where $\xi$ is between $0$
and $\beta\sigma W/n$.
Similarly,
\begin{eqnarray*}
\hspace*{-4pt}&&E\bigl(\bigl(W-W'\bigr)^2|\Sigma\bigr)
\\
\hspace*{-4pt}&&\qquad= \frac{4}{\sigma^2} E\bigl(I\bigl(\sigma_I=1,
\sigma_I'=-1\bigr)+I\bigl(\sigma_I=-1,
\sigma_I'=1\bigr)|\Sigma\bigr)
\\
\hspace*{-4pt}&&\qquad=\frac{2(1-m^2)}{\sigma^2}+O(1)\frac{W}{n \sigma}+O\biggl(\frac{1}{n
\sigma^2}
\biggr)+O\biggl(\frac{I(S-2<an\mbox{ or } S+2>bn)}{\sigma^2}\biggr).
\end{eqnarray*}
Therefore, recall that $\sigma^2=n\frac{1-m^2}{1-(1-m^2)\beta}$,
\[
\bigl|E(D|W)-1\bigr|\le O\biggl(\frac{1}{\sqrt{n}}\biggr) \bigl(1+|W|\bigr).
\]
For $R$, with $\delta_2=O(1/\sqrt{n})$,
\[
\bigl|E(R|W)\bigr|\le\delta_2 \bigl(1+W^2\bigr),
\]
and if $c_1$ is chosen such that $\delta_2 |W|\le1/2$, the second
alternative of (\ref{c3}) is satisfied with $\alpha=1/2$. Thus from
Theorem~\ref{t1}, we have the following moderate deviation
result for $W$:
\[
\frac{ P(W \geq x) }{1-\Phi(x)} = 1+ O(1) \bigl(1+x^3\bigr) \frac
{1}{\sqrt{n}}
\]
for $0\le x\le n^{1/6}$. This completes the proof of (\ref{p4a}) and
(\ref{p5a}).

\subsection{\texorpdfstring{Proof of Proposition \protect\ref{t71}}{Proof of Proposition 4.5}}\label{sec6.4}

Since $(1-\Phi(x)) \geq{ 1 \over2 (1+x)} e^{-x^2/2}$ for $x\geq0$,
(\ref{t71a}) becomes trivial if $ x\gamma\geq1/8$. Thus we can assume
%
%e6.36 #&#
\begin{equation} \label{t71-00}
x \gamma\leq1/8.
\end{equation}
Let $f=f_x$ be
the Stein solution to equation (\ref{stein2}).
Let
$W^{(i)} = W- \xi_i$ and
$
K_i(t) = E\xi_i ( I\{ 0 \leq t \leq\xi_i\} - I \{\xi_i \leq t \leq
0\})$.
It is known that
[see, e.g., (2.18) in \citet{CS05}]
\[
EWf(W) = \sum_{i=1}^n E \int
_{-\infty}^\infty f'\bigl(W^{(i)} +
t\bigr) K_i(t) \,dt.
\]
Since $\int_{-\infty}^\infty K_i(t) \,dt = E\xi_i^2$, we have
%
%e6.37 #&#
\begin{eqnarray}\label{t71-1}
&&P(W \geq x) - \bigl(1-\Phi(x)\bigr)
\nonumber
\\
&&\qquad= EWf(W) - Ef'(W)
\nonumber
\\
&&\qquad= \sum_{i=1}^n E \int
_{-\infty}^\infty\bigl( f'
\bigl(W^{(i)} +t\bigr) - f'(W)\bigr) K_i(t) \,dt
\nonumber\\[-8pt]\\[-8pt]
&&\qquad= \sum_{i=1}^n E \int
_{-\infty}^\infty\bigl( \bigl(W^{(i)} + t\bigr)
f\bigl(W^{(i)}+t\bigr) - W f(W)\bigr) K_i(t) \,dt
\nonumber
\\
&&\qquad\quad{} + \sum_{i=1}^n E \int
_{-\infty}^\infty\bigl( I\bigl\{W^{(i)} + t \geq
x\bigr\} - I\{W \geq x\}\bigr) K_i(t) \,dt
\nonumber
\\
&&\qquad:= R_1 + R_2.\nonumber
\end{eqnarray}
It suffices to show that
%
%e6.38 #&#
\begin{equation} \label{r1}
|R_1| \leq C \bigl(1+x^3\bigr) \gamma\bigl(1- \Phi(x)
\bigr) e^{x^3 \gamma}
\end{equation}
and
%
%e6.39 #&#
\begin{equation}\label{R2}
|R_2| \leq C \bigl(1+x^2\bigr) \gamma\bigl(1- \Phi(x)
\bigr) e^{x^3 \gamma}.
\end{equation}
To estimate $R_1$, let $g(w) = (wf(w))'$.
It is easy to see that
%
%e6.40 #&#
\begin{equation}\label{r1-1}
R_1 = \sum_{i=1}^n E \iint_{\xi_i}^t g\bigl(W^{(i)} +s\bigr) \,ds
K_i(t) \,dt.
\end{equation}
%
%and
%$$
%g(w) =
%(\sqrt{2\pi}(1+w^2) e^{w^2/2} (1-\Phi(w)) - w)\Phi(x), & w\geq x\\
%(\sqrt{2\pi}(1+w^2) e^{w^2/2} \Phi(w) + w)(1-\Phi(x)), & w < x.
%$$
%Observing that (see (5.4) in \citet{CS01})
%$$
%0 \leq\sqrt{2\pi} (1+w^2) e^{w^2/2} ( 1- \Phi(w)) - w \leq{ 2 \over
%1 + w^3} \mbox{for} w \geq0,
%$$
%}
By (\ref{t1-11}) and (\ref{t1-012}), following
the proof of (\ref{t1-12}), we have
%
%e6.41 #&#
\begin{eqnarray}
\label{t71-2}
&&
E g\bigl(W^{(i)} +s\bigr)
\nonumber\\
&&\qquad= Eg\bigl(W^{(i)}+s\bigr)I\bigl\{W^{(i)}+s\geq x\bigr\} +
Eg\bigl(W^{(i)}+s\bigr)I\bigl\{W^{(i)}+s\leq0 \bigr\}
\nonumber
\\
&&\qquad\quad{}+ Eg\bigl(W^{(i)} +s \bigr)I\bigl\{0< W^{(i)} + s <
x\bigr\}
\nonumber
\\
&&\qquad\leq\frac{ 2 }{1+x^3} P\bigl(W^{(i)} +s\geq x\bigr) + 2 \bigl(1-
\Phi(x)\bigr)P\bigl(W^{(i)}+s\leq0\bigr)
\nonumber
\\
&&\qquad\quad{} + \sqrt{2\pi} \bigl(1-\Phi(x)\bigr) \nonumber\\
&&\qquad\quad\hspace*{11pt}{}\times E \bigl\{ \bigl
(1+\bigl(W^{(i)}+s
\bigr)^2\bigr)e^{(W^{(i)}+s)^2/2} I\bigl\{0< W^{(i)}+s < x\bigr
\} \bigr\}
\\
&&\qquad\leq\frac{ 2 }{1+x^3} P\bigl(W^{(i)}\geq x-s\bigr) + 2 \bigl(1-
\Phi(x)\bigr)P\bigl(W^{(i)}+s \leq0\bigr)
\nonumber
\\
&&\qquad\quad{} - \sqrt{2\pi} \bigl(1-\Phi(x)\bigr)\int_0^x
\bigl(1+y^2\bigr)e^{y^2/2} \,dP\bigl(W^{(i)}+s > y
\bigr)
\nonumber
\\
&&\qquad\leq\frac{ 2 }{1+x^3} P\bigl(W^{(i)}\geq x-s\bigr) + 2 \bigl(1-
\Phi(x)\bigr)P\bigl(W^{(i)}+s \leq0\bigr)
\nonumber\\
&&\qquad\quad{} + \sqrt{2\pi} \bigl(1-\Phi(x)\bigr)P\bigl(W^{(i)}+s
>0\bigr) +
\sqrt{2\pi} \bigl(1-\Phi(x)\bigr) J(s)
\nonumber
\\
&&\qquad\leq\frac{ 2 }{1+x^3} P\bigl(W^{(i)} \geq x-s\bigr) + \sqrt
{2\pi}
\bigl(1-\Phi(x)\bigr) + \sqrt{2\pi} \bigl(1-\Phi(x)\bigr) J(s),
\nonumber
\end{eqnarray}
where
%
%e6.42 #&#
\begin{equation} \label{t71-3}
J(s) = \int_0^x \bigl(3y+y^3
\bigr) e^{y^2/2} P\bigl(W^{(i)}+s >y\bigr) \,dy.
\end{equation}
Clearly, for $0 < t \leq x$
\begin{eqnarray*}
Ee^{t \xi_j} & =& 1 + t^2 E\xi_j^2 / 2
+ \sum_{k=3}^\infty\frac{ ( t \xi_j)^k }{ k!}
\\
& \leq& 1 + t^2 E\xi_j^2 /2 +
\frac{ t^3 }{6} E |\xi_j|^3 e^{t |\xi
_j|}
\\
& \leq& \exp\biggl( t^2 E\xi_j^2/2 +
\frac{ x^3 }{6} E |\xi_j|^3 e^{
x |\xi_j|} \biggr)
\end{eqnarray*}
and hence
%
%e6.43 #&#
\begin{equation}\label{t71-5}
Ee^{t(W^{(i)}+s)} \leq\exp\biggl( t^2/2 + x |s| +
\frac{ x^3 }{6} \gamma\biggr) \qquad\mbox{for } 0 \leq t \leq x.
\end{equation}
By (\ref{t71-5}), following the proof of Lemma~\ref{l22} yields
%
%e6.44 #&#
\begin{equation} \label{t71-6}
J(s) \leq C \bigl(1+x^3\bigr) e^{ x^3 \gamma+ x |s|}.
\end{equation}
Noting that (\ref{t71-5}) also implies that
\begin{eqnarray*}
P\bigl( W^{(i)} \geq x -s\bigr) & \leq& e^{-x^2}
Ee^{x ( W^{(i)}+s)} \leq\exp\bigl( - x^2/2 + x|s| + x^3
\gamma\bigr)
\\
& \leq& ( 1+ x) \bigl( 1- \Phi(x)\bigr) \exp\bigl( x|s| + x^3 \gamma
\bigr),
\end{eqnarray*}
we have
\[
Eg\bigl(W^{(i)}+s\bigr) \leq C \bigl(1+x^3\bigr) \bigl( 1-
\Phi(x)\bigr) e^{ x^3 \gamma+ x |s|}
\]
and therefore by (\ref{r1-1}),
%
%e6.45 #&#
\begin{eqnarray} \label{t71-7}\quad
|R_1| & \leq& \sum_{i =1}^n
E \int_{-\infty}^\infty\biggl|\int_{\xi_i}^t
g\bigl( W^{(i)} +s\bigr) \,ds \biggr| K_i(t) \,dt
\nonumber
\\
& \leq& C \bigl( 1+ x^3\bigr) \bigl(1-\Phi(x)\bigr) e^{x^3 \gamma}
\sum_{i=1}^n E\int_{-\infty}^\infty
\bigl(|t| e^{x|t|} + |\xi_i| e^{x|\xi_i|}\bigr)
K_i(t) \,dt
\\
& \leq& C \bigl(1+x^3\bigr) \gamma\bigl(1- \Phi(x)\bigr)
e^{x^3 \gamma}.\nonumber
\end{eqnarray}
This proves (\ref{r1}).

As to $R_2$, we apply an exponential concentration inequality of
\citet{Shao2010} [see Theorem 2.7 in \citet{Shao2010}]: for
$a\geq0$ and $b \geq0$,
\begin{eqnarray*}
&&
P\bigl( x - a \leq W^{(i)} \le x +b\bigr)
\\
&&\qquad\leq C e^{x \gamma+ x a - x^2} \bigl( (\gamma+ b+a) E\bigl|W^{(i)}\bigr|
e^{x W^{(i)}} + \bigl( E e^{2x W^{(i)}}\bigr)^{1/2} \exp\bigl( -
\gamma^{-2}/32\bigr) \bigr)
\\
&&\qquad\leq C e^{x \gamma+ x a - x^2} \bigl( (\gamma+ b+a) \bigl(EW^{(i)}
e^{x W^{(i)}} + 1\bigr)\\
&&\qquad\quad\hspace*{57.2pt}{}+ \bigl( E e^{2x W^{(i)}}\bigr)^{1/2} \exp
\bigl( - \gamma^{-2}/32\bigr) \bigr)
\\
&&\qquad\leq C e^{x \gamma+ x a - x^2} \bigl((\gamma+ b+a)  (1+x)
e^{x^2/2 + x^3 \gamma} +
e^{ x^2 + x^3\gamma} \exp\bigl( - \gamma^{-2}/32\bigr) \bigr)
\\
&&\qquad\leq C e^{x^3 \gamma+ x a - x^2/2} \bigl( (\gamma+ b+a) (1+x)
+ \exp\bigl(
x^2/2 - \gamma^{-2}/32\bigr) \bigr)
\\
&&\qquad\leq C \bigl(1-\Phi(x)\bigr) e^{x^3 \gamma+ xa} \bigl( (\gamma
+ b+a)
\bigl(1+x^2\bigr) + \exp\bigl( x^2 - \gamma^{-2}/32
\bigr) \bigr).
\end{eqnarray*}
Here we use the fact that $E W^{(i)} e^{x W^{(i)}} \leq x e^{x^2/2 +
x^3 \gamma}$, by following the proof of
(\ref{t71-5}). Therefore
\begin{eqnarray*}
R_2 & \leq& \sum_{i=1}^n E
\int_{-\infty}^\infty P\bigl( x - \xi_i
\leq W^{(i)} \leq x -t | \xi_i\bigr) K_i(t)
\,dt
\\
& \leq& C \bigl(1-\Phi(x)\bigr) e^{x^3 \gamma} \sum
_{i=1}^n \int_{-\infty}^\infty
\bigl\{ \bigl(1+x^2\bigr) E\bigl(\gamma+ |t|+|\xi_i|\bigr)
e^{x |\xi_i|}\\
&&\hspace*{159pt}{} + \exp\bigl( x^2 - \gamma^{-2}/32\bigr)
\bigr\} K_i(t) \,dt
\\
& \leq& C \bigl( 1- \Phi(x)\bigr) e^{x^3 \gamma} \bigl( \bigl(1+x^2
\bigr) \gamma+ \exp\bigl( x^2 - \gamma^{-2}/32\bigr) \bigr)
\\
& \leq& C \gamma\bigl(1+x^2\bigr) \bigl( 1- \Phi(x)\bigr)
e^{x^3 \gamma}
\end{eqnarray*}
by (\ref{t71-00}).
Similarly, the above bound holds for $-R_2$. This proves (\ref{R2}).

\section*{Acknowledgments}
We would like to thank the Associate Editor and the referees for their
helpful comments and suggestions which have significantly improved the
presentation of this paper.

%suskaldyti doi

% imsref loaded by lrinkeviciute, 2012-07-20 15:55:55
% imsref loaded by lrinkeviciute, 2012-07-23 08:40:42

\printaddresses

\end{document}